\documentclass[a4paper,11pt]{article}

\usepackage[utf8]{inputenc}
\usepackage[T1]{fontenc}
\usepackage{geometry}
\usepackage{amsmath}
\usepackage{amssymb}
\usepackage[all]{xy}
\usepackage{mathenv}
\usepackage[dvips]{color}
\usepackage{float}
\usepackage{psfrag}
\usepackage{pgf,tikz}
\usetikzlibrary{arrows}
\usepackage{amsthm}
\usepackage[french,english]{babel}
\usepackage{amssymb}
\usepackage{esint}
\usepackage{ upgreek }

\theoremstyle{plain}
\newtheorem{thm}{Theorem}[section]
\newtheorem*{thm*}{Theorem}

\newtheorem{prop}[thm]{Proposition}
\newtheorem{lem}[thm]{Lemma}

\newtheorem*{prop*}{Proposition}

\theoremstyle{definition}
\newtheorem{deft}[thm]{Definition}

\theoremstyle{remark}

\newtheorem{rmk}[thm]{Remark}

\numberwithin{equation}{section}

\def\ve{\varepsilon}

\def\c{\mathcal{C}}

\def\d{{\mathcal{C}_{0}^{\infty}}(\mathbb{R}^{2})}

\def\R{\mathbb{R}}

\def\N{\mathbb{N}}

\def\d{\mathop{}\mathopen{}\mathrm{d}}
\def\ind{\mathbf{1}}

\def\XXint#1#2#3{{\setbox0=\hbox{$#1{#2#3}{\displaystyle\int}$}\vcenter{\hbox{$#2#3$}}\kern-.5\wd0}}

\newcommand{\dint}{\displaystyle\int}
\newcommand{\dsum}{\displaystyle\sum}

\newcommand\hone{{\mathcal H}^1}
\newcommand{\be}{\begin{enumerate}}
\newcommand{\ee}{\end{enumerate}}
\newcommand{\op}[1]{\operatorname{#1}}

\title{Uniform estimates for a Modica-Mortola type approximation of branched transportation}
\author{Antonin Monteil\footnote{Laboratoire de Math\'ematiques d'Orsay, Universit\'e Paris-Sud 11, B\^at. 425, 91405, Orsay, France (e-mail: antonin.monteil@math.u-psud.fr)}}

\begin{document}
\selectlanguage{english}
\maketitle
\begin{abstract}
Models for branched networks are often expressed as the minimization of an \mbox{energy} $M^\alpha$ over vector measures concentrated on $1$-dimensional rectifiable sets with a divergence constraint. We study a Modica-Mortola type approximation $M^\alpha_\ve$, introduced by Edouard Oudet and Filippo Santambrogio, which is defined over $H^1$ vector measures. These energies induce some pseudo-distances between $L^2$ functions obtained through the minimization problem $\min \{M^\alpha_\ve (u)\;:\;\nabla\cdot  u=f^+-f^-\}$. We prove some uniform estimates on these pseudo-distances which allow us to \mbox{establish} a $\Gamma$-convergence result for these energies with a divergence constraint.
\end{abstract}

\section{Introduction}
Branched transportation is a classical problem in optimization: it is a variant of the Monge-Kantorovich optimal transportation theory in which the transport cost for a mass $m$ per unit of length is not linear anymore but sub-additive. More precisely, the cost to transport a mass $m$ on a length $l$ is considered to be proportional to $m^\alpha l$ for some $\alpha\in ]0,1[$. As a result, it is more efficient to transport two masses $m_1$ and $m_2$ together instead of transporting them separately. For this reason, an optimal pattern for this problem has a ``graph structure'' with branching points. Contrary to what happens in the Monge-Kantorovich model, in the setting of branched transportation, an optimal structure cannot be described only using a transport plan, giving the correspondence between origins and destinations, but we need a model which encodes all the trajectories of mass particles. \\
\indent Branched transportation theory is motivated by many structures that can be found in the nature: vessels, trees, river basins\dots\ Similarly, as a consequence of the economy of scale, large roads are proportionally cheaper than large ones and it follows that the road and train networks also present this structure. Surprisingly the theory has also had theoretical applications: recently, it has been used by F. Bethuel in \cite{Bethuel:2014} so as to study the density of smooth maps in Sobolev spaces between manifolds.

Branched transportation theory was first introduced in the discrete framework by E. N. Gilbert in \cite{Gilbert:1967} as a generalization of the Steiner problem. In this case an admissible structure is a weighted graph composed of oriented edges of length $l_i$ on which some mass $m_i$ is flowing. The cost associated to it is then $\sum_i l_i m_i^\alpha$ and it has to be minimized over all graphs which transport some given atomic measure to another one.
More recently, the branched transportation problem was generalized to the continuous framework by Q. Xia in \cite{Xia:2003} by means of a relaxation of the discrete energy (see also \cite{Xia:2004}). Then, many other models and generalizations have been introduced (see \cite{Maddalena:2003} for a Lagrangian formulation, see also \cite{Bernot:2005}, \cite{Bernot:2008}, \cite{Bernot:2009} for different generalizations and regularity properties.). In this paper, we will concentrate on the model with a divergence constraint, due to Q. Xia. However, this is not restrictive since all these models have been proved to be equivalent (see \cite{Bernot:2009} and \cite{Pegon:2015}).

In this model, a transport path is represented as a vector measure $u$ on some open set $\Omega\subset\R^d$ such that $\nabla \cdot u=\mu^+-\mu^-$ for two probability measures $\mu^+$ and $\mu^-$. Then the energy of $u$ is defined as $M^\alpha (u)=\int_M \theta^\alpha d\hone $ if $u$ is a vector measure concentrated on a rectifiable $1$-dimensional set $M$ on which $u$ has multiplicity $\theta$ w.r.t. the Hausdorff measure (see \cite{Bernot:2009} for more details). In this framework, $u$ must be considered as the momentum (the mass $\theta$ times the velocity) of a particle at some point. Then $(\nabla\cdot  u) (x)$ represents the difference between incoming and outcoming mass at each point $x$.

In this paper, we are interested in some approximation of branched transportation proposed by E. Oudet and F. Santambrogio few years ago in \cite{Oudet:2011} and which has interesting numerical applications. This model was inspired by the well known scalar phase transition model proposed by L. Modica and S. Mortola in \cite{Modica:1977}. Given $u\in H^1 (\Omega,\R^d)$ for some bounded open subset $\Omega\subset \R^d$, E. Oudet and F. Santambrogio introduced the following energy:
$$M^\alpha_\ve (u)=\ve^{-\gamma_1}\int_\Omega |u|^\beta +\ve^{\gamma_2} \int_\Omega |\nabla u|^2,$$
where $\beta\in (0,1)$ and $\gamma_1,\gamma_2>0$ are some exponents depending on $\alpha$ (see \eqref{malphaepsilon}). If $u$ does not belong to the set $H^1 (\Omega)$, the value of $M^\alpha_\ve$ is taken as $+\infty$. 

We recall the heuristic which shows why $M^\alpha_\ve$ is an approximation of $M^\alpha$ (see \cite{Oudet:2011}): assume that $\mu^-$ (resp. $\mu^+$) is a point source at $S_1$ (resp. $S_2$) with mass $m$. Then, it is clear that the optimal path for $M^\alpha$ between these two measures is the oriented edge $S=(S_1,S_2)$ of length $l$ with a mass $m$ flowing on it. We would like to approximate this structure, seen as a vector measure $u$ concentrated on $S$, by some $H^1$ vector fields $v$ which are more or less optimal for $M^\alpha_\ve$. What we expect is that $v$ looks like a convolution of $u$ with a kernel $\rho$ depending on $\ve$ and $m$: $v=u\ast \rho_R$ for some $R=R(\ve,m)$, where 
 \begin{equation}\label{rhor}
 \rho_R(x)=R^{-d}\rho (R^{-1}x)
 \end{equation}
for some fixed smooth and compactly supported radial kernel $\rho\in \c^\infty_c (\R^d)$. Then the support of $v$ is like a strip of width $R$ around $S$ so that $| v|$ is of the order of $m/R^{d-1}$ and $|\nabla v|$ is of the order of $m/R^d$. This gives an estimate of $M^\alpha_\ve (v)$ like
\begin{equation}\label{heuristic}
M^\alpha_\ve (v)\simeq \ve^{-\gamma_1} R^{d-1} (m/R^{d-1})^\beta l+\ve^{\gamma_2} R^{d-1} (m/R^d)^2 l.
\end{equation}
With our choice for the exponents $\gamma_1$, $\gamma_2$ and $\beta$, the optimal choice for $R$ is
\begin{equation}\label{R}
R=\ve^{\gamma}m^{\frac{1-\gamma}{d-1}},
\end{equation}
where 
\begin{equation}\label{defgamma}\gamma=\frac{2}{2d-\beta (d-1)}=\frac{\gamma_2}{d+1}.
\end{equation} 
This finally leads to $M^\alpha_\ve (v)\simeq m^\alpha$ as expected.

It was proved in \cite{Oudet:2011} that, at least in two dimensions, the energy sequence $(M^\alpha_\ve)_{\ve>0}$ $\Gamma$-converges to the branched transportation functional $c_0 M^\alpha$ for some constant $c_0$ and for some suitable topology (see Theorem \ref{gammaconvergenceoudet} page \pageref{gammaconvergenceoudet}). This result has been interestingly applied to produce a numerical method. However, rather than a $\Gamma$-convergence result on $M_\ve^\alpha$ we would need to deal with the functionals $\overline{M}_\ve^\alpha$, obtained by adding a divergence constraint: it should be shown that $\overline{M}^\alpha_\ve (u):=M^\alpha_\ve (u)+I_{\nabla\cdot  u=f_\ve}$ $\Gamma$-converges to $c_0\overline{M}^\alpha (u):=c_0 M^\alpha_\ve (u)+ I_{\nabla\cdot  u=\mu^+-\mu^-}$, where $f_\ve \in L^2$ is some suitable approximation of $\mu^+-\mu^-$ and $I_A(u)$ is the indicator function in the sense of convex analysis that is $0$ whenever the condition is satisfied and $+\infty$ otherwise. Even if this property was not proved in \cite{Oudet:2011}, the effectiveness of the numerical simulations made the authors think that it actually holds true. Note that an alternative using a penalization term was proposed in \cite{Santambrogio:2010} to overcome this difficulty.

In section \ref{mathematicalsetting} we recall Xia's formulation of branched transportation and its approximation $M^\alpha_\ve$ introduced by E. Oudet and F. Santambrogio. The longest part of this paper, section \ref{localestimatesection}, is devoted to a local estimate which gives a bound on the minimum value $d^\alpha_\ve (f^+,f^-):=\min\{M^\alpha_\ve (u)\;:\; \nabla\cdot  u=f\}$ depending on $\|f\|_{L^1}$, $\|f\|_{L^2}$ and $\op{diam}(\Omega)$ (see Theorem \ref{localestimate} page \pageref{localestimate}). In section \ref{comparisonsection}, we deduce a comparison between $d^\alpha_\ve$ and the Wasserstein distance with an ``error term'' involving the $L^2$ norm of $f^+-f^-$. As an application of this inequality, in the last section, we will prove the following $\Gamma$-convergence result which was lacking in \cite{Oudet:2011}
\begin{thm}\label{gammacv}
Let $(f_\varepsilon)_{\varepsilon>0}\subset L^2(\Omega)$ be a sequence weakly converging to $\mu$ as measures when $\varepsilon\to 0$. Assume that the sequence $(f_\varepsilon)_{\varepsilon>0}$ satisfies
\begin{equation*}\label{hypponfeps}
\int_\Omega f_\varepsilon(x)\d x=0\quad\text{and}\quad\ve^{\gamma_2}\|f_\ve\|_{L^2}^2\underset{\ve\to 0}{\longrightarrow} 0.
\end{equation*}
There exists a constant $c_0>0$ such that the functional sequence $(\overline{M}^\alpha_{\ve})_{\ve>0}$ $\Gamma$-converges to $c_0\overline{M}^\alpha$ as $\ve\to 0$. Moreover $c_0$ is the minimum value for the minimizing problem \eqref{chan}.
\end{thm}
This answers the Open question 1 in \cite{Santambrogio:2010,Oudet:2011} and validates their numerical method.

\section{Mathematical setting.\label{mathematicalsetting}}
\paragraph{The branched transportation energy}
In all what follows, we will use the model proposed by Q. Xia (see \cite{Xia:2003} and \cite{Xia:2004}):

Let $d\geq 1$ be an integer and $\Omega$ be some open and bounded subset of $\R^d$. Let us denote by $\mathcal{M}_{div}(\Omega)$ the set of finite vector measures on $\overline{\Omega}$ such that their divergence is also a finite measure:
$$\mathcal{M}_{div}(\Omega):=\left\{u\text{ measure on }\overline{\Omega}\text{ valued in }\R^d\;:\;\|u\|_{\mathcal{M}_{div}(\Omega)}<+\infty\right\},$$
where $\|u\|_{\mathcal{M}_{div}(\Omega)}:=|u|(\overline{\Omega})+|\nabla\cdot  u|(\overline{\Omega})$ with
$$|u|(\overline{\Omega}):=\sup \left\{ \int_{\overline{\Omega}} \psi\cdot \d u \;:\; \psi\in\c(\overline{\Omega},\,\R^d),\, \|\psi\|_{L^\infty}\leq 1 \right\}$$
and, similarly,
$$|\nabla\cdot  u|(\overline{\Omega}):=\sup \left\{ \int_{\overline{\Omega}} \nabla\varphi \cdot \d u\;:\; \varphi\in\c^1(\overline{\Omega},\R),\, \|\varphi\|_{L^\infty}\leq 1 \right\}.$$
In all what follows, $\nabla\cdot  u$ has to be thought in the weak sense, i.e. $\int\varphi\nabla\cdot  u = -\int \nabla\varphi\cdot\d u$ for all $\varphi \in\c^1(\overline{\Omega})$. Since we do not ask $\varphi$ to vanish at the boundary, $\nabla\cdot  u$ may contain possible parts on $\partial \Omega$ which are equal to $u\cdot n$ when $u$ is smooth, where $n$ is the external unit normal vector to $\partial\Omega$. In other words, $\nabla\cdot  u$ is the weak divergence of $u\ind_\Omega$ in $\R^d$, where $\ind_\Omega$ is the classical indicator function of $\Omega$, equal to $1$ on $\Omega$ and $0$ elsewhere. From now on, the notation $\ind_X$ for the classical indicator function of a set $X$ and $I_X$ for the indicator function in the sense of convex analysis (equal to $1$ inside and $+\infty$ outside) will be used. $\mathcal{M}_{div}(\Omega)$ is endowed with the topology of weak convergence on $u$ and on its divergence: i.e. $u_n\overset{\mathcal{M}_{div}(\Omega)}{\longrightarrow}u$ if $u_n\rightharpoonup u$ and $\nabla\cdot  u_n \rightharpoonup\nabla\cdot  u$ weakly as measures.\\

Given $0<\alpha<1$, the energy of branched transportation can be represented as follows for measures $u\in\mathcal{M}_{div}(\Omega)$:
\begin{equation}\label{defmalpha}
M^\alpha(u)=
\begin{cases} \int_M \theta^\alpha \d\hone &if $u$ can be written as $u = U(M,\theta,\xi)$,\\                                                                +\infty &otherwise,
\end{cases}
\end{equation}
 where $U(M,\theta,\xi)$ is the rectifiable vector measure $u=\theta\xi\cdot\hone_{|M}$ with density $\theta\xi$ with respect to the $\hone-$Hausdorff measure on the rectifiable set $M$. The real multiplicity is a measurable function $\theta: M\to\R^+$ and the orientation $\xi:M\to S^{d-1}\subset\R^d$ is such that $\xi (x)$ is tangential to $M$ for $\hone$-a.e. $x\in M$.
 
Given two probability measures $\mu^+$ and $\mu^-$ on $\overline{\Omega}$, the problem of branched transportation consists in minimizing $M^\alpha$ under the constraint $\nabla\cdot  u=\mu^+-\mu^-$:
\begin{equation}\label{minmalpha}\inf\left\{ M^\alpha (u) \;:\; u\in \mathcal{M}_{div}(\overline{\Omega})\quad\text{and}\quad\nabla\cdot  u=\mu^+-\mu^-\right\}.\end{equation}
Note that, if $\mu^\pm (\partial\Omega)=0$, the divergence constraint implies a Neumann condition on $u$: $u\cdot n=0$ on $\partial\Omega$. 
\paragraph{Functionals $M^\alpha_\ve$} For the minimum value in \eqref{minmalpha} to be finite whatever $\mu^+$ and $\mu^-$ in the set of probability measures, we will require $\alpha$ to be sufficiently close to $1$. More precisely, we make the following assumption:
\begin{equation}\label{irrigationprop}
1-\frac{1}{d}<\alpha<1.
\end{equation}
Q. Xia has shown in \cite{Xia:2003} that, under this assumption, there exists at least one vector measure $u\in\mathcal{M}_{div}(\Omega)$ such that $M^\alpha (u)<+\infty$.\\

We are interested in the following approximation of $M^\alpha$ which was introduced in \cite{Oudet:2011}: for all $u\in\mathcal{M}_{div}(\Omega)$ and for all open subset $\omega\subset \Omega$,
\begin{equation}\label{malphaepsilon}
M^\alpha_\ve(u,\omega):=
\begin{cases}
\ve^{-\gamma_1}\dint_\omega |u(x)|^\beta \d x + \ve^{\gamma_2}\dint_\omega |\nabla u(x)|^2 \d x & if $u\in H^1(\omega)$\\
+\infty & otherwise,
\end{cases}
\end{equation}
where $\beta$, $\gamma_1$ and $\gamma_2$ are three exponents depending on $\alpha$ and $d$ through:
$$\beta=\frac{2-2d+2\alpha d}{3-d+\alpha (d-1)}$$
and
$$
\gamma_1=(d-1)(1-\alpha)\quad\text{and}\quad\gamma_2=3-d+\alpha (d-1).
$$
Note that inequality $1-1/d<\alpha <1$ implies that $0<\beta<1$. When $\omega=\Omega$, we simply write 
$$M^\alpha_\ve (u,\Omega)=:M^\alpha_\ve (u).$$
We point out the 2-dimensional case where $M^\alpha_\ve$ rewrites as 
\begin{equation}\label{malphaepsilon2d}
M^\alpha_\ve(u)=\ve^{\alpha-1}\dint_\Omega |u(x)|^\beta \d x + \ve^{\alpha+1}\dint_\Omega |\nabla u(x)|^2 \d x,
\end{equation}
where $\beta= \frac{4\alpha-2}{\alpha+1}$. 

Given two densities $f^+,f^-\in L^2_+(\Omega):=\{f\in L^2(\Omega)\;:\; f\geq 0\}$ such that $\int f^+=\int f^-$, we are interested in minimizing $M^\alpha_\ve(u)$ under the constraint $\nabla\cdot  u =f^+-f^-$:
\begin{equation}\label{minmalpha}\inf\left\{ M^\alpha_\ve (u) \;:\; u\in H^1(\Omega)\quad\text{and}\quad\nabla\cdot  u=f^+-f^-\right\}.\end{equation}
The classical theory of calculus of variation shows that this infimum is actually a minimum. A natural question that arises is then to understand the limit behavior for minimizers of these problems when $\ve$ goes to $0$. A classical tool to study this kind of problems is the theory of $\Gamma$-convergence which was introduced by De Giorgi in \cite{DeGiorgi:1975}. For the definition and main properties of $\Gamma$-convergence, we refer to \cite{DalMaso:1993} and \cite{Braides:2002}. In particular, if $M^\alpha_\ve$ $\Gamma$-converges to some energy functional $M_0^\alpha$ and if $(u_\ve)$ is a sequence of minimizers for $M^\alpha_\ve$ admitting a subsequence converging to $u$, then, $u$ is a minimizer for $M_0^\alpha$. By construction of $M^\alpha_\ve$, we expect that, up to a subsequence, $M^\alpha_\ve$ $\Gamma$-converges to $c_0M^\alpha$. In the two dimensional case, we have the following $\Gamma$-convergence theorem proved in \cite{Oudet:2011}:
\begin{thm}\label{gammaconvergenceoudet}
Assume that $d=2$ and $\alpha \in (1/2,1)$. Then, there exists a constant $c>0$ such that $(M^\alpha_\ve)_{\ve>0}$ $\Gamma$-converges to $cM^\alpha$ in $\mathcal{M}_{div}(\Omega)$ when $\ve$ goes to $0$.
\end{thm}
Nevertheless, this does not imply the $\Gamma$-convergence of $M^\alpha_\ve (u)+I_{\nabla\cdot  u=f^+-f^-}$ to 
$M^\alpha_\ve (u)+ I_{\nabla\cdot  u=f^+-f^-}$. Indeed, the $\Gamma$-convergence is stable under the addition of continuous functionals but not l.s.c. functionals. Consequently, we cannot deduce, from this theorem, the behavior of minimizers for \eqref{minmalpha}. For instance, it is not clear that there exists a recovery sequence $(u_\ve)$, i.e. $u_\ve$ converges to $u$ in $\mathcal{M}_{div}(\Omega)$ and $M^\alpha_\ve (u_\ve)$ converges to $M^\alpha (u)$ as $\ve\to 0$, with prescribed divergence $\nabla\cdot u_\ve =f^+-f^-$. To this aim, we require some estimates on these energies and this is the purpose of this paper. 
\paragraph{Distance of branched transportation}
We remind our hypothesis $1-1/d<\alpha<1$. In \cite{Xia:2003}, Q. Xia has remarked that, as in optimal transportation theory, $M^\alpha$ induces a distance $d^\alpha$ on the space $\mathcal{P}(\overline{\Omega})$ of probability measures on $\overline{\Omega}$:
$$d^\alpha (\mu^+,\mu^-)=\inf\left\{M^\alpha(u)\;:\; u\in\mathcal{M}_{div}(\Omega)\quad\text{such that}\quad\nabla \cdot u=\mu^+-\mu^-\right\},$$
for all $\mu^+,\mu^- \in\mathcal{P}(\overline{\Omega})$. Thanks to our assumption $\alpha>1-1/d$, $d^\alpha$ is finite for all $\mu^\pm\in \mathcal{P}(\overline{\Omega})$ and it induces a distance on the set $\mathcal{P}(\overline{\Omega})$ which metrizes the topology of weak convergence of measures. Actually, $d^\alpha$ has a stronger property which is a comparison with the Wasserstein distance:

\begin{prop} \label{comparisonprop} Let $\mu^+$ and $\mu^-$ be two probability measures on $\overline{\Omega}$. We denote by $W_{p}$ the Wasserstein distance associated to the cost $(x,y)\to |x-y|^p$ for $p\geq 1$. Then, one has
$$W_{1/\alpha}(\mu^+,\mu^-)\leq d^\alpha (\mu^+,\mu^-)\leq C\; W_{1}(\mu^+,\mu^-)^{1-d(1-\alpha)},$$
for a constant $C>0$ only depending on $d$, $\alpha$ and the diameter of $\Omega$.
\end{prop}
\noindent We refer to \cite{Morel:2007} for a proof of this property (see also \cite{Bernot:2009}, and \cite{Brasco:2011} for an alternative proof) and \cite{Villani:2003}, \cite{Santambrogio:2015} for the definition and main properties of the Wasserstein distance. In the same way, we define $d^\alpha_\ve$ as follows:
\begin{equation}\label{defdalphaeps}
d^\alpha_\ve (f^+,f^-)=\inf\left\{M^\alpha_\ve(u)\;:\; u\in H^1(\R^d)\quad\text{such that}\quad\nabla \cdot u=f^+-f^-\right\},
\end{equation}
where $f^+,f^-\in L^2_+(\Omega)$ satisfy $\int_\Omega f^+=\int_\Omega f^-$. Although $d^\alpha$ is a distance, it is not the case for $d^\alpha_\ve$ which does not satisfy the triangular inequality. Actually, because of the second term involving $|\nabla u|^2$, $M_{\ve}^{\alpha}$ is not subadditive. However, for $u_1,\dots , u_n$ in $\mathcal{M}_{div}(\Omega)$, the inequality $|\nabla u_1+\dots +\nabla u_n|^2\leq n\{|\nabla u_1|^2+\dots+|\nabla u_n|^2\}$ implies
\begin{equation*}\label{notsubadditive}
M^{\alpha}_{\ve}\left(\displaystyle\sum_{i=1}^{n}u_i\right)\leq n\sum_{i=1}^{n}M^{\alpha}_{\ve}(u_i).
\end{equation*}
In particular, $d^\alpha_\ve$ is a pseudo-distance in the sense that the three properties in the following proposition are satisfied:
\begin{prop}\label{propnotsubadditive}
Let $f^+$, $f^-$ and $f_1$,\dots, $f_n$ be $L^2$ densities, i.e. $L^2$ non negative functions whose integral is equal to 1. Then one has
\begin{enumerate}
\item $d^{\alpha}_{\ve}(f^+,f^-)=0$ implies $f^+=f^-$,
\item $d^{\alpha}_{\ve}(f^+,f^-)=d^{\alpha}_{\ve}(f^-,f^+)$,
\item $d^{\alpha}_{\ve}(f_0,f_{n})\leq n\;\big[ d^{\alpha}_{\ve}(f_0,f_1)+d^{\alpha}_{\ve}(f_1,f_2)+\dots+d^{\alpha}_{\ve}(f_{n-1},f_{n})\big]$.
\end{enumerate}
\end{prop}

\section{Local estimate\label{localestimatesection}}
We remind our assumption \eqref{irrigationprop} which insures that $d^\alpha(\mu^+,\mu^-)$ is always finite. Our goal is to prove that $d^\alpha_\ve$ enjoys a property similar to the following one.
\begin{prop}\label{localestimatedalpha}
Let $Q_0=(0,L)^d\subset\R^d$ be a cube of side length $L>0$. There exists some constant $C>0$ only depending on $d$ and $\alpha$ such that for all non negative Borel finite measure $\mu$ of total mass $\theta>0$,
$$d^{\alpha}(\mu,\theta\delta_{0})\leq C\ \theta^{\alpha} L,$$
where $\delta_{0}$ is the Dirac measure at the point $c_{Q_0}$, the center of $Q_0$.
\end{prop}
Since $d^\alpha_\ve (f^+,f^-)$ is only defined on $L^2$ functions $f^\pm$, to do so, we first have to replace $\theta \delta_0$ by some kernel which concentrates at the origin when $\ve$ goes to $0$. Let $\rho\in\c_c^1(B,\R^+)$ be a radial non negative function such that $\int_{\R^d}\rho=1$, where $B\subset\R^d$ is the unit ball centered at the origin, and define $\rho_{\theta,\ve}:=\rho_R$ as in \eqref{rhor}, where
\begin{equation*}\label{rthetaepsilon}R=:R_{\theta,\ve}=\ve^\gamma\theta^{\frac{1-\gamma}{d-1}}.\end{equation*}
Here, we recall that $R$ and $\gamma=\frac{\gamma_2}{d+1}$ were introduced in \eqref{defgamma}. Let $Q$ be a cube in $\R^d$ centered at some point $c_Q\in\R^d$ and $f\in L_+^2(Q)$ be a density such that $\int_Q f=:\theta_Q$. Then, we will denote by $\rho_Q$ the kernel $\theta \rho_{\theta,\ve}$ refocused at $c_Q$ with a small abuse of notation (indeed, $\rho_Q$ also depends on $f$):
$$\rho_Q(x)=\theta_Q\rho_{\theta_Q,\ve}(x-c_Q).$$
The main result of this section is the following theorem
\begin{thm}[Local estimate]\label{localestimate}
Let us set $Q_0=(0,L)^d$ for some $L>0$. There exists $C>0$ only depending on $\alpha$, $\rho$ and $d$ such that for all $f\in L_+^2(Q_0)$ with $\int_{Q_0}f=\theta$, we have 
\begin{itemize}
\item If $\op{supp}\rho_{Q_0}\subset Q_0$ then, there exists $u\in H^1_0(Q_0)$ such that $\nabla\cdot  u=f-\rho_{Q_0}$ and
$$d^{\alpha}_{\ve}(f,\rho_{Q_0})\leq M^\alpha_\ve (u)\leq C\big\{\theta^\alpha L
+\ve^{\gamma_2} \|f\|_{L^2}^2\big\}\quad\text{and}\quad\|u\|_{L^1}\leq C\,L\, \theta.$$
\item Otherwise, there exists $u\in H^1_0(\widetilde{Q}_0)$ such that
$$d^\alpha_\ve(f,\rho_{Q_0})\leq M^\alpha_\ve (u)\leq C\ve^{\gamma_2} \|f\|_{L^2}^2\quad\text{and}\quad\|u\|_{L^1}\leq C\,L\, \theta,$$
where $\widetilde{Q}_0=2\op{supp}\rho_{Q_0}:=B(c_{Q_0}, 2 R_{\theta,\ve})$.
\end{itemize}
\end{thm}

\begin{rmk}
The Dirichlet term, $\ve^{\gamma_2} \|f\|_{L^2}^2$, in the estimates above is easily understandable. Indeed, if $\ve$ is very large so that one can get rid of the first term in the energy $M^\alpha_\ve$, then, one can use a classical Dirichlet type estimate, that is Theorem \ref{dirichlet} below. On the contrary, for $\ve$ very small, the $\Gamma$-limit result on $M^{\alpha}_{\ve}$ tells us that these energies are close to $M^{\alpha}$ so that it is natural to hope a similar estimate as the one for $M^{\alpha}$: that is to say an estimate from above by $\theta^{\alpha} L$ (see \cite{Bernot:2009}).
\end{rmk}

The main difficulty to prove Theorem \ref{localestimate} is the non subadditivity of the pseudo-distances $\lambda_\ve ^{\alpha}$. Indeed, our proof is based on a dyadic construction used by Q. Xia in \cite{Xia:2003} to prove Proposition \ref{localestimatedalpha} (see also \cite{Bernot:2009}). This gives a singular vector measure $u$ which is concentrated on a graph. Since $M^\alpha_\ve$ contains a term involving the $L^2$ norm of $\nabla u$, we have to regularize $u$ by taking a convolution with the kernel $\rho_{\theta,\ve}$ on each branch of the graph ($\theta$ being the mass traveling on it). Unfortunately in this way, two different branches are no longer disjoints.

It is useful to see that we have a first candidate for the minimization problem \eqref{defdalphaeps}. This candidate is of the form $u=\nabla \phi$, where $\phi$ is the solution of the Dirichlet problem
\begin{equation}\label{dirichletPB}\left\{\begin{array}{lcll}
\Delta \phi &=&f^+-f^-&\text{in}\quad Q,\\
\phi&=&0 &\text{on}\quad\partial Q.
\end{array}\right.
\end{equation}
Then $u=\nabla\phi$ satisfies $\nabla\cdot u=f^+-f^-$ in $Q$ and $u(x)\in\R \, n \,$ a.e. on $\partial Q$ where $n$ stands for the external unit normal vector to $\partial Q$. Alternatively, one could consider Neumann homogeneous boundary conditions for $\phi$ rather than Dirichlet boundary conditions. Then, one would obtain $u(x)\cdot n=0$ a.e. on $\partial Q$. Theorem \ref{dirichlet} below gives a better result in the sense that the candidate $u$ vanishes at the boundary:

\begin{thm}\label{dirichlet}
Let $Q_0=(0,L)^d$ be a cube of side length $L>0$. There exists $C>0$ only depending on $d$ such that for all $ f\in L_0^2(Q_0)$, there exists $u\in H^1_0(Q_0,\R^2)$ solving $\nabla\cdot  u=f$ and satisfying $\|u\|_{L^1(Q_0)} \leq C L\,\|f\|_{L^1(Q_0)}$ together with
$$\|u\|_{H_0^1(Q_0)}:=\left(\dint_{Q_0}|\nabla u|^2\right)^{1/2}\leq C\,\|f\|_{L^2(Q_0)},$$
where $L_0^2(Q_0)=\left\{f\in L^2(Q_0)\;:\;\int_{Q_0}f(x)\d x=0\right\}$. 
\end{thm}
For a proof of this theorem, see, for instance, Theorem 2 in \cite{Bourgain:2003}: the only difference with Theorem \ref{dirichlet} is that we add the estimate $\|u\|_{L^1(Q_0)} \leq C L\,\|f\|_{L^1(Q_0)}$ which can be easily obtained following the proof of J. Bourgain and H. Brezis. The corresponding property formulated on a Lipschitz bounded connected domain $\Omega$ is also true (see Theorem 2' in \cite{Bourgain:2003}) except that the constant $C$ could depend on $\Omega$ in this case.

Of course, this candidate is usually not optimal for \eqref{defdalphaeps} and this does not allow for a good estimate because of the first term in the definition of $M^{\alpha}_{\ve}$. For this reason, we have to use the dyadic construction of Q. Xia up to a certain level (``diffusion level'') from which we simply use Theorem \ref{dirichlet}. 

\subsection{Dyadic decomposition of $Q_0$ and ``diffusion level'' associated to f}
Let us call ``dyadic descent'' of $Q_0=(0,L)^d$ the set $\mathcal{Q}=\bigcup_{j\geq 0}\mathcal{Q}_j$, where $\mathcal{Q}_j$ is the $j^{th}$ ``dyadic generation'':
$$\mathcal{Q}_j=\left\{(x_1,\dotsc,x_d)+ 2^{-j}Q_0  \;:\; x_i\in\{k2^{-j}L\;:\; 0\leq k\leq2^j-1\}\quad\text{for}\quad i=1,\dotsc,d\right\}.$$
Note that $\op{Card}(\mathcal{Q}_j)=2^{jd}$. For each $Q\in\mathcal{Q}$, let us define
\begin{itemize}
\item $\mathcal{D}(Q)$: the descent of $Q$, the family of all dyadic cubes contained in $Q$. 
\item $\mathcal{A}(Q)$: the ancestry of $Q$, the family of all dyadic cubes containing $Q$. 
\item $\mathcal{C}(Q)$: the family of children of $Q$ composed of the $2^d$ biggest dyadic cubes strictly included in $Q$.
\item $F(Q)$: the father of $Q$, the smallest dyadic cube strictly containing $Q$.
\end{itemize}
We now remind the dyadic construction described in \cite{Xia:2003} which irrigates $f$ from a point source. We first introduce some notations: fix a function $f\in L^2_+(Q_0)$ with integral $\theta$ and let $Q\in\mathcal{Q}$ be a dyadic cube centered at $c_Q\in \R^d$. Then we introduce $\theta_Q$ the mass associated to the cube $Q$ as
$$\theta_Q=\int_Q f\quad .$$
If $\theta_Q\neq 0$, we also define the kernel associated to $Q$ through
\begin{equation}\label{defrhoqbar}\overline{\rho}_Q(x)=\rho_R(x),\end{equation}
where $\rho_R$ is defined in \eqref{rhor} for
\begin{equation*}\label{rq}R=R_Q := \ve^\gamma \theta_Q^\frac{1-\gamma}{d-1}\quad ,\quad \gamma=\frac{\gamma_2}{d+1}\quad .\end{equation*}
Here $\gamma$ was defined in Define also the weighted recentered kernel by
\begin{equation}\label{defrhoq}\rho_Q(x)=\theta_Q\overline{\rho}_Q(x-c_Q)\end{equation}
if $\theta_Q\neq 0$ and $\rho_Q(x)=0$ otherwise. Lastly, we introduce the point source associated to the cube $Q$ as
$$\mathcal{S}_Q:=\theta_Q \times \text{Dirac measure at point } c_Q.$$
We are now able to construct a vector measure $X$ such that $M^\alpha (X)<+\infty$. First define the measures $X_Q$ as below:
\begin{equation}\label{muq}
X_Q=\sum_{Q'\in\mathcal{C}(Q)} \theta_{Q'}\; n_{Q'}\;\hone_{|[c_Q,c_{Q'}]},
\end{equation}
where $\displaystyle n_{Q'}=\frac{c_{Q'}-c_{Q}}{\| c_{Q'}-c_{Q}\|}$. Then, we have
$$\nabla\cdot X_Q=\displaystyle\sum_{Q'\in\mathcal{C}(Q)} \mathcal{S}_{Q'}-\mathcal{S}_{Q}$$
and the energy estimate
$$M^\alpha(X_Q)\leq 2^{d-2}  \theta_Q^\alpha \op{diam}(Q),$$
where $\op{diam}(Q)$ stands for the diameter of $Q$. Finally, the measure $X=\sum_{Q\in\mathcal{Q}}X_Q$ solves $\nabla\cdot  X =f-\mathcal{S}_{Q_0}$ and satisfies
$$M^\alpha (X)\leq C\theta^\alpha \op{diam}(Q_0).$$
Indeed, it is enough to apply the following lemma with $\lambda=\alpha$:
\begin{lem}\label{malphafinite}
Let $Q\in \mathcal{Q}$ and $\lambda\in ]1-1/d,1]$. There exists a constant $C=C(\lambda, d)$ such that
$$\sum_{Q'\in \mathcal{D}(Q)} \theta_{Q'}^\lambda \op{diam}(Q')\leq C \theta_Q^\lambda \op{diam}(Q).$$
\end{lem}

\begin{proof}
Let $j_0\geq 0$ be such that $Q\in\mathcal{Q}_{j_0}$. The definition of $\mathcal{D}(Q)$, the Jensen inequality and the fact that $d-1-\lambda d <0$ give
\begin{align*}
\dsum_{Q'\in \mathcal{D}(Q)} \theta_{Q'}^\lambda \op{diam}(Q')&=\displaystyle\sum_{j\geq 0}2^{-j}\op{diam}(Q)\displaystyle\sum_{Q'\in\mathcal{D}(Q)\cap\mathcal{Q}_{j_0+j}}\theta_{Q'}^{\lambda}\\
&\leq  \op{diam}(Q) \dsum_{j\geq 0}2^{-j} 2^{jd} \left( 2^{-jd} \sum_{Q'\in\mathcal{D}(Q)\cap\mathcal{Q}_{j_0+j}}\theta_{Q'}\right)^{\lambda}\\
&\leq  \theta_Q^\lambda\op{diam}(Q) \dsum_{j\geq 0}2^{j(d-1-\lambda d)}\\
&\leq C\theta_Q^{\lambda} \op{diam}(Q).\qedhere
\end{align*}
\end{proof}
Now, the idea is to replace each term in \eqref{muq} by its convolution with the kernel $\overline{\rho}_{Q'}$. Unfortunately, this will make appear extra divergence terms around each node. We have to modify $X$ so as to make this extra divergence vanish using, for instance, Theorem \ref{dirichlet}. Furthermore, we cannot follow the construction for all generations $j\geq 1$, otherwise the ``enlarged edges'' (convolution of a dyadic edge and the kernel $\rho_{\theta,\ve}$) may overlap. This is the reason why we introduce the following definition: 
\begin{deft}[``Diffusion level'']\label{diffstate} For a cube $Q_0$ and $f\in L^2_+(Q_0)$ we define the set $\mathcal{D}(Q_0,f)$ or $\mathcal{D}(f) \subset\mathcal{Q}$ as the maximal element for the inclusion in the set 
$$\Lambda=\left\{D\subset\mathcal{Q}\;:\;  \forall Q\in D,\; \mathcal{A}(Q)\cup\mathcal{C}(F(Q))\subset D\quad\text{and}\quad\op{supp}\rho_Q\subset Q\right\}.$$
If $\Lambda=\emptyset $, that is $\op{supp}\rho_{Q_0}\nsubseteq Q_0$, we take the convention $\mathcal{D}(f)=\emptyset$. For all $x\in Q_0$, define also the ``generation index'' of $x$ associated to $f$ as
$$j(f,x)=\max\left\{j\;:\; \exists Q\in \mathcal{D}(f)\cap\mathcal{Q}_j,\; x\in Q\right\}\in \N\cup\{\pm\infty\},$$
where the convention $\max (\emptyset )=-\infty$ has been used.
\end{deft}
In this way, each cube in $\mathcal{D}(f)$ contains the support of its associated kernel. Moreover, if $Q$ is an element of $\mathcal{D}(f)$, then all its ancestry and its brothers (i.e. elements of the set $\mathcal{C}(F(Q))$) are elements of $\mathcal{D}(f)$. $\mathcal{D}(f)$ can be constructed by induction as follows: first take $j=0$ and $\mathcal{D}(f)=\emptyset$. If $\op{supp}\rho_{Q_0}\subset Q_0$ then add $Q_0$ to the set $\mathcal{D}(f)$ and $j$ is replaced by $j+1$. For all cubes $Q$ in $\Lambda\cap\mathcal{Q}_{j-1}$: if all cubes $Q'\in \mathcal{C}(Q)\subset\mathcal{Q}_j$ are such that their associated kernels are supported on $Q'$ then $\mathcal{D}(f)$ is replaced by $\mathcal{D}(f)\cup\mathcal{C}(Q)$. If $\mathcal{D}(f)$ has been changed at this stage $j$ is replaced by $j+1$ and the preceding step is reiterated. This process is repeated for $j\geq 1$ until it fails.

Let $\mathcal{D}_{min}(f)$ be the set of all cubes in $\mathcal{D}(f)$ which are minimal for the inclusion. If $\mathcal{D}_{min}(f)\neq \emptyset$, we also define
$$D(f)=\displaystyle\bigcup_{Q\in\mathcal{D}_{min}(f)}Q\quad .$$
Note that this is actually a disjoint union: two distinct cubes in $\mathcal{D}_{min}(f)$ are  disjoint. Indeed, for $Q$, $Q'\in\mathcal{D}_{min}(f)\subset\mathcal{Q}$, either $Q\cap Q'=\emptyset$ or $Q$ and $Q'$ are comparable: $Q\subset Q'$ or $Q'\subset Q$. In the last case, since $Q$ and $Q'$ are minimal, we deduce that $Q=Q'$.

Moreover, it is not difficult to see that, if $\mathcal{D}_{min}(f)\neq \emptyset$, then $D(f)=\{ x\in Q_0\;:\; j(f,x)\text{ is finite} \}$ and also that $\overline{f}(x)=0$ whenever $j(f,x)=+\infty$, where $\overline{f}$ is the precise representative of $f$ (i.e. the limit of the mean values of $f$ on small cubes). Indeed, assume that $Q\in\mathcal{D}(f)$ is a cube of side length $L_Q$. Then, by definition, $\op{supp}\rho_Q\subset Q$ and for some constant $C$ depending on $\rho$ and for $\nu=\frac{1-\gamma}{d-1}$, one has $\ve^\gamma\theta_Q^\nu\leq CL_Q$ and so
$$\fint_Q f:=L_Q^{-d}\theta_Q\leq \ve^{-\gamma/\nu}L_Q^{1/\nu-d}.$$
Since $1/\nu-d=\frac{(d-1)(\alpha d-d+1)}{d+1}$ is positive, we deduce that $L_Q$ cannot be arbitrarily small if there exists $x\in Q$ such that $\overline{f}(x)>0$. Moreover, if $f(x)\geq \eta$ a.e. for some $\eta>0$, then there exists some constant $C_\eta>0$ depending on $\eta$, $\ve$, $d$ and $\alpha$ such that 
\begin{equation}\label{estimationofj}
\forall Q\in \mathcal{D}(f),\, L_Q\geq C_\eta.
\end{equation}
In particular, one can deduce that $\mathcal{D}_{min}(f)=\emptyset$ if and only if $\mathcal{D}(f)=\emptyset$ or $f(x)=0$ a.e. Indeed, if $\mathcal{D}(f)=\emptyset$, then it is clear that $\mathcal{D}_{min}(f)=\emptyset$. Conversely, assume that $\mathcal{D}_{min}(f)\neq\emptyset$ (i.e. $Q_0\in\mathcal{D}(f)$) and that there exists $x\in Q_0$ such that $\overline{f}(x)>0$. Since $\bigcup_{Q\in \mathcal{D}(f)} \partial Q$ is negligible for the Lebesgue measure, one can assume that $x\in \bigcup_{Q\in \mathcal{D}(f)} Q$. Then $0\leq j(f,x)<+\infty$ and so there exists a minimal cube $Q\in\mathcal{D}(f)$ containing $x$. Then $Q\in\mathcal{D}_{min}(f)$. Indeed, if $Q'\in\mathcal{D}(f)$ and $Q'\subsetneq Q$, then $\mathcal{A}(Q)\subset \mathcal{D}(f)$ and there exists $Q''\in \mathcal{A}(Q)$ such that $Q''\subsetneq Q$ and $x\in Q''$ which is a contradiction.

We are now able to define two approximations of $f$ which are useful for our problem. The first is a dyadic approximation of $f$ by an atomic measure,
\begin{equation*}
\Lambda_{\ve}f=
\begin{cases}
\displaystyle\sum_{Q\in\mathcal{D}_{min}(f)}\mathcal{S}_Q&if $\mathcal{D}_{min}(f)\neq\emptyset$,\\
\mathcal{S}_{Q_0}&otherwise,
\end{cases}
\end{equation*}
where we recall the definition of $\mathcal{S}_Q:=\theta_Q \delta_{c_Q}$. We also define an approximation in $H^1(Q_0)$,
\begin{equation*}\label{diffusionstate}
\lambda_{\ve}f=
\begin{cases}
\dsum_{Q\in\mathcal{D}_{min}(f)}\rho_Q&if $\mathcal{D}_{min}(f)\neq\emptyset$,\\
\rho_{Q_0}&otherwise,
\end{cases}
\end{equation*}
where $\rho_Q$ is defined in \eqref{defrhoq}. The following result shows in which sense $\lambda_\ve f$ is an approximation of $f$ and justifies the term ``diffusion level''. Indeed, this proposition indicates that we get a good estimate by using a local diffusion from $\lambda_\ve  f$ to $f$, i.e. minimizing $\int_Q |\nabla u|^2$ over the constraint $\nabla\cdot u=\lambda_\ve  f-f$ for all $Q\in\mathcal{D}_{min}(f)$ (see Theorem \ref{dirichlet}).
\begin{prop}\label{distdiffstate}
There exists a constant $C>0$ depending on $d$ and $\rho$ such that for all $f\in L^2_+(Q_0)$,
$$
 d^{\alpha}_{\ve}(\lambda_\ve f,f)+d^{\alpha}(\Lambda_\ve f,f)\leq C\;\ve^{\gamma_2}\|f\|_{L^2(Q_0)}^2 .
$$
More precisely, if $\op{supp}\rho_{Q_0}\subset Q_0$, there exists $u\in H^1_0 (Q_0)$ such that $\nabla \cdot u=f-\lambda_\ve  f$ as well as
$$
M_\ve^\alpha(u)\leq C\, \ve^{\gamma_2}\|f\|^2_{L^2}\quad \text{and}\quad \|u\|_{L^1}\leq C\, \op{diam}(Q_0)\|f\|_{L^1}.
$$
If $\op{supp}\rho_{Q_0}\nsubseteq Q_0$ the same estimates hold but the condition $u\in H^1_0(Q_0)$ has to be replaced by $u\in H^1_0 (\widetilde{Q}_0)$, where $\widetilde{Q}_0$ is a cube containing $Q_0$ and $\op{supp}\rho_{Q_0}$.
\end{prop}

\begin{proof}
First assume that $\op{supp}\rho_{Q_0}\subset Q_0$ i.e. $Q_0\in \mathcal{D}(f)$. If $\mathcal{D}_{min}(f)= \emptyset$, then $f(x)=0$ a.e. and the proposition is trivial. Hence, one can assume that $\mathcal{D}_{min}(f)\neq \emptyset$. Then $f$ is supported on $D(f)$ and $\mathcal{D}_{min}(f)=:\{Q_i\}_{i\in I}$ is a finite or countable partition of $D(f)$. Denote for simplicity $D_i:=\op{diam}(Q_i)$, $f_{i}:=f\ind_{Q_i}$ (restriction of $f$ to $Q_i$), $\theta_i:=\theta_{Q_i}$ and $\rho_i:=\rho_{Q_i}=\theta_i\,\rho_{R_i}$ for $i\in I$, where 
$$R_i:=R_{Q_i}=\ve^\gamma\theta_i^{\frac{1-\gamma}{d-1}}.$$ 
Since $Q_i$ is minimal in $\mathcal{D}(f)$, we deduce that, for some constants $C,C'>0$,
\begin{equation}\label{dri}
C' R_i\leq D_i \leq C R_i.
\end{equation}
Indeed, the first inequality follows from the fact that $\op{supp}\rho_i\subset Q_i$ and $\op{diam}(\op{supp}\rho_i)=cR_i$ for some constant $c$ depending on $\rho$. For the second inequality observe that, since $Q_i$ is minimal, there exists $Q\in \mathcal{C}(Q_i)$ such that $\op{supp}\rho_Q\nsubseteq Q$ and hence $R_Q\geq c'\op{diam}(Q)=c'/2D_i$ for some constant $c'>0$ depending on $\rho$. Since $\theta_Q\leq \theta_{Q_i}=\theta_i$, one has $R_Q\leq R_i$ and the second inequality follows.

Now, Theorem \ref{dirichlet} allows us to find $u_i\in H^1_0(Q_i)$ such that $\nabla\cdot  u_i=g_i$, $\|u_i\|_{H^1(Q_i)}\leq C\,\|g_{i}\|_{L^2(Q_i)}$ and $\|u_i\|_{L^1(Q_i)}\leq C\, D_i\|g_i\|_{L^1(Q_i)}$, where $g_i:=f_{i}-\rho_{i}$. Since $u_i$ vanishes at $\partial Q_i$, one can extend $u_i$ by $0$ out of $Q_i$ to get a function in $H^1 (\R^d)$: for the sake of simplicity, this function is still denoted by $u_i$. Consequently, $u=\sum_i u_i$ belongs to $H^1_0 (Q_0)$ and $\nabla\cdot  u=f-\lambda_\ve  f$. It remains to estimate $M^\alpha_\ve (u)$ and $\|u\|_{L^1(Q_0)}$. First of all, 
$$\|u\|_{L^1(Q_0)}\leq \sum_i \|u_i\|_{L^1(Q_i)}\leq C\, \op{diam}(Q_0)\sum_i \|g_i\|_{L^1(Q_i)}$$ 
and the inequality $\|g_i\|_{L^1(Q_i)}\leq 2\theta_i$ leads to $\|u\|_{L^1}\leq 2C\, \op{diam}(Q_0)\|f\|_{L^1}$ as required. 

Let us compute the $L^2$-norm of $\rho_i$:
$$\|\rho_i\|_{L^2(Q_i)}^2=\theta_i^2\; \|\rho_{R_{i}}\|^2_{L^2(Q_i)}=\theta_i^2 R_i^{-d} \|\rho\|_{L^2(Q_i)}^2= C\theta_i^2 R_i^{-d}.$$
By a Cauchy-Schwarz inequality, 
 \begin{equation}\label{cauchyschwarz}\theta_i^2=\left(\dint_{Q_i} f_i\right)^2\leq |Q_i| \|f_i\|_{L^2(Q_i)}^2=D_i^d \|f_i\|_{L^2(Q_i)}^2\end{equation}
 which, together with \eqref{dri}, gives
$$\|\rho_i\|_{L^2(Q_i)}^2\leq C\, R_i^d \|f_i\|^2_{L^2(Q_i)}R_i^{-d}= C\, \|f_i\|^2_{L^2(Q_i)}.$$
Since $\|u_i\|_{H^1(Q_i)}\leq C \|f_i-\rho_i\|_{L^2(Q_i)}$, we get $\|u_i\|_{H^1(Q_i)}\leq C\|f_{i}\|_{L^2(Q_i)}$. Now, because the energy $M^\alpha_\ve$ is local and since each $u_i$ is supported on $Q_i$, one has
\begin{equation*}\label{2terms}
M^{\alpha}_{\ve}(u)=\displaystyle\sum_{i=1}^n M^{\alpha}_{\ve}(u_i)
=\displaystyle\sum_{i=1}^n\left( \ve^{-\gamma_1} \int_{Q_i}|u_i|^{\beta}+\ve^{\gamma_2}\int_{Q_i}|\nabla u_i|^2\right)
\end{equation*}
By construction of $u_i$, one has 
$$\dint_{Q_i}|\nabla u_i|^2\leq \|u_i\|_{H^1(Q_i)}^2\leq C\|g_i\|_{L^2(Q_i)}^2\leq 2C\left(\|\rho_i\|_{L^2(Q_i)}^2+\|f_i\|_{L^2(Q_i)}^2\right)\leq C'\|f_i\|^2_{L^2(Q_i)}\quad .$$
It remains to estimate the first term. First of all, we use the H\"older and Poincar\'e inequalities as follows:
$$
\displaystyle\int_{Q_i}|u_i|^{\beta}\leq|Q_i|^{1-\beta/2}\left(\dint_{Q_i}|u_i|^{2}\right)^{\beta/2}
\leq D_i^{d-d\beta/2}\left(D_i^2\dint_{Q_i}|\nabla u_i|^{2}\right)^{\beta/2}
\leq D_i^{\nu} \|f_i\|^{\beta}_{L^2(Q_i)},
$$
where $\nu=\beta+d-\frac{d\beta}{2}$. In view of \eqref{cauchyschwarz} and \eqref{dri}, we have
$$D_i\leq C R_i= C \ve^{\gamma} \theta_i^{\frac{1-\gamma}{d-1}} \leq  C \ve^{\gamma} (D_i^{\frac{d}{2}} \|f_i\|_{L^2(Q_i)})^{\frac{1-\gamma}{d-1}} $$
and, introducing $\delta :=1-\frac{d(1-\gamma)}{2(d-1)}$,
\begin{equation}\label{di}D_i^{\delta}\leq C\ve^\gamma \|f_i\|^{\frac{1-\gamma}{d-1}}_{L^2}.\end{equation}
Finally, since $-\gamma_1+\frac{\gamma\nu}{\delta}=\gamma_2$ and $\beta+\frac{\nu(1-\gamma)}{\delta(d-1)}=2$, we get
$$\ve^{-\gamma_1}\dint_{Q_i}|u_i|^\beta\leq C \ve^{-\gamma_1+\frac{\gamma\nu}{\delta} }\|f_i\|_{L^2(Q_i)}^{\beta+\frac{\nu(1-\gamma)}{\delta(d-1)}}=C \ve^{\gamma_2}\|f_i\|_{L^2(Q_i)}^2.$$
The proof of the second inequality is quite similar but easier:
$$d^{\alpha}(\Lambda_\ve f,f)\leq\displaystyle\sum_{i=1}^n d^{\alpha}(\mathcal{S}_{Q_i},f_i)\leq\displaystyle\sum_{i=1}^n \theta_i^{\alpha}D_i\quad .$$
Once again, applying \eqref{cauchyschwarz} and then \eqref{di}, we get 
$$d^\alpha(\Lambda_\ve  f,f)\leq C\, \ve^{\gamma_2} \|f\|^2_{L^2}.$$
In the case where $\op{supp}\rho_{Q_0}\nsubseteq Q_0$, i.e. $R_{Q_0}:=\ve^\gamma \theta_Q^\frac{1-\gamma}{d-1}\geq CL$ ($L$ being the side length of $Q_0$ and $C$ a constant depending on $\rho$), the proof is the same. Indeed, we just apply Theorem \ref{dirichlet} to $g=f-\rho_{Q_0}$ and the same computations as above lead to the same result.
\end{proof}

\subsection{Proof of Theorem \ref{localestimate}}
Let $Q_0=(0,L)^d$, $L>0$ and $f\in L_+^2(Q_0)$ with $\int_{Q_0}f=\theta$. In the case where $\op{supp}\rho_{Q_0}\nsubseteq Q_0$, Theorem \ref{localestimate} is a particular case of Proposition \ref{distdiffstate}. Consequently, one can assume that $\op{supp}\rho_{Q_0}\subset Q_0$ i.e. $Q_0\in \mathcal{D}(f)$. In the case where $\mathcal{D}(f)=\{Q_0\}$, one has $\lambda_\ve  f=\rho_{Q_0}$ and Theorem \ref{localestimate} is a consequence of Proposition \ref{distdiffstate} as well. For this reason, one can assume that $\mathcal{C}(Q_0)\subset \mathcal{D}(f)$. Moreover, up to replacing $f$ by $f+\eta$ for some small constant $\eta>0$ and passing to the limit when $\eta\to 0$, one can assume that $\mathcal{D}(f)$ is finite. Indeed, in view of \eqref{estimationofj}, $\mathcal{D}(f+\eta)$ is finite since for all $Q\in \mathcal{D}(f+\eta)$, $\op{diam}(Q)\geq C_\eta >0$.

Our aim is to prove that there exists $C>0$ only depending on $\alpha$, $d$ and $\rho$ such that
 \begin{equation*}\label{localestim}
d^{\alpha}_{\ve}(f,\rho_{Q_0})\leq C\big\{\theta^\alpha L
+\ve^{\gamma_2} \|f\|_{L^2(Q_0)}^2\big\}.
\end{equation*}
The idea of the proof is to approximate the vector field $X=\sum X_Q$ of the previous section (see \eqref{muq}) by a vector field in $H^1$ using the kernel $\rho$. In this part, we will use the notations of the previous section: in particular, the definition of $\mathcal{D}(f)$ in Definition \ref{diffstate}, the measures $X_Q$ in \eqref{muq} and $X =\sum_{Q\in\mathcal{D}(f)} X_Q$.

Since $\mathcal{C}(Q_0)\subset \mathcal{D}(f)$, we can construct the regularized vector field $Y$ by the formula
\begin{equation*}\label{v}
Y=\sum_{\substack{Q\in\mathcal{D}(f) \\Q\neq Q_0}}Z_Q,
\end{equation*}
where, for all $Q\in\mathcal{D}(f)$ such that $Q\neq Q_0$ (see Figure \ref{sausage}),
\begin{equation}\label{uq}
Z_{Q}:=\theta_Q\; n_{Q}\;\overline{\rho}_{Q}\ast \hone_{|[c_{F(Q)},c_Q]},\end{equation}
$n_Q$ being the normalized vector $n_Q=\frac{c_{Q}-c_{F(Q)}}{\| c_{Q}-c_{F(Q)}\|}$ and $\overline{\rho}_Q$ being defined in \eqref{defrhoqbar}. 

\begin{figure}[H]
\centering
\begin{tikzpicture}[line cap=round,line join=round,>=triangle 45,x=4.5cm,y=4.5cm]
\clip(-1.1,-1.1) rectangle (1.1,1.1);
\draw (-1.,1.)-- (1.,1.);
\draw (1.,1.)-- (1.,-1.);
\draw (1.,-1.)-- (-1.,-1.);
\draw (-1.,-1.)-- (-1.,1.);
\draw(0.5,0.5) circle (0.763675323681cm);

\draw [dotted] (0.5,0.5) circle (0.254558441227cm);
\draw [dotted] (0.46,0.54)-- (0.71,0.79);
\draw [dotted] (0.79,0.71)-- (0.54,0.46);
\draw [dotted] (0.46,0.46)-- (0.21,0.71);
\draw [dotted] (0.54,0.46)-- (0.29,0.21);
\draw [dotted] (0.29,0.79)-- (0.54,0.54);
\draw [dotted] (0.21,0.29)-- (0.46,0.54);

\draw (0.38,0.62)-- (-0.12,0.12);
\draw (0.12,-0.12)-- (0.62,0.38);
\draw(-0.511666420184,0.512557428874) circle (0.591269427819cm);
\draw(-0.5,-0.5) circle (0.154263880307cm);
\draw(0.5,-0.5) circle (0.443599395946cm);
\draw (-0.0233328403686,0.025114857748)-- (-0.52424023019,-0.47575976981);
\draw (0.0242402301905,-0.0242402301905)-- (-0.47575976981,-0.52424023019);
\draw (0.430295079777,-0.569704920223)-- (-0.0697049202231,-0.0697049202231);
\draw (0.0697049202231,0.0697049202231)-- (0.569704920223,-0.430295079777);
\draw (-0.604575447277,0.419648401781)-- (-0.0929090270931,-0.0929090270931);
\draw (-0.418757393091,0.605466455967)-- (0.0929090270931,0.0929090270931);

\draw [dotted] (0.439654200638,0.439654200638)-- (0.689654200638,0.189654200638);
\draw [dotted] (0.560345799362,0.560345799362)-- (0.810345799362,0.310345799362);

\draw [->,line width=1pt] (-0.141505994761,0.141505994761) -- (-0.294299275179,0.294299275179);
\draw [->,line width=1pt] (-0.141505994761,-0.141505994761) -- (-0.294299275179,-0.294299275179);
\draw [->,line width=1pt] (0.141505994761,-0.141505994761) -- (0.294299275179,-0.294299275179);
\draw [->,line width=1pt] (0.141505994761,0.141505994761) -- (0.294299275179,0.294299275179);

\draw [->,dotted] (0.420430632755,0.420430632755) -- (0.320855398352,0.320855398352);
\draw [->,dotted] (0.579569367245,0.420430632755) -- (0.679144601648,0.320855398352);
\draw [->,dotted] (0.420430632755,0.579569367245) -- (0.320855398352,0.679144601648);
\draw [->,dotted] (0.579569367245,0.579569367245) -- (0.679144601648,0.679144601648);

\draw (-1,1) node[anchor=north west] {$Q_1$};
\draw (1,1) node[anchor=north east] {$Q_2$};
\draw (-1,-1) node[anchor=south west] {$Q_3$};
\draw (1,-1) node[anchor=south east] {$Q_4$};

\draw (-0,1) node[anchor=north east] {$Q$};

\draw (-0.254695574082,0.52) node[anchor=north west] {$Z_{Q_1}$};
\draw (0.22,0.25) node[anchor=east] {$Z_{Q_2}$};
\draw (-0.51,-0.12) node[anchor=north west] {$Z_{Q_3}$};
\draw (0.41,-0.177021342608) node[anchor=north west] {$Z_{Q_4}$};

\draw (-0.5,0.5) node {$\rho_{Q_1}$};
\draw (0.691690334656,0.62) node[anchor=north west] {$\rho_{Q_2}$};
\draw (-0.5,-0.5) node[anchor=south east] {$\rho_{Q_3}$};
\draw (0.5,-0.5) node {$\rho_{Q_4}$};

\draw [shift={(-0.511666420184,0.512557428874)}] plot[domain=0.785398163397:3.92699081699,variable=\t]({1.*0.131393206182*cos(\t r)+0.*0.131393206182*sin(\t r)},{0.*0.131393206182*cos(\t r)+1.*0.131393206182*sin(\t r)});
\draw [shift={(0.25,0.75)},dotted]  plot[domain=0.785398163397:3.92699081699,variable=\t]({1.*0.0565685424949*cos(\t r)+0.*0.0565685424949*sin(\t r)},{0.*0.0565685424949*cos(\t r)+1.*0.0565685424949*sin(\t r)});
\draw [shift={(0.75,0.75)},dotted]  plot[domain=-0.785398163397:2.35619449019,variable=\t]({1.*0.0565685424949*cos(\t r)+0.*0.0565685424949*sin(\t r)},{0.*0.0565685424949*cos(\t r)+1.*0.0565685424949*sin(\t r)});
\draw [shift={(0.75,0.25)},dotted]  plot[domain=-2.35619449019:0.785398163397,variable=\t]({1.*0.0853418478894*cos(\t r)+0.*0.0853418478894*sin(\t r)},{0.*0.0853418478894*cos(\t r)+1.*0.0853418478894*sin(\t r)});
\draw [shift={(0.,0.)}] plot[domain=2.35619449019:5.49778714378,variable=\t]({1.*0.169705627485*cos(\t r)+0.*0.169705627485*sin(\t r)},{0.*0.169705627485*cos(\t r)+1.*0.169705627485*sin(\t r)});
\draw [shift={(0.,0.)}] plot[domain=0.785398163397:3.92699081699,variable=\t]({1.*0.0985776435436*cos(\t r)+0.*0.0985776435436*sin(\t r)},{0.*0.0985776435436*cos(\t r)+1.*0.0985776435436*sin(\t r)});
\draw [shift={(0.25,0.25)},dotted]  plot[domain=2.35619449019:5.49778714378,variable=\t]({1.*0.0565685424949*cos(\t r)+0.*0.0565685424949*sin(\t r)},{0.*0.0565685424949*cos(\t r)+1.*0.0565685424949*sin(\t r)});
\draw [shift={(0.,0.)}] plot[domain=-2.35619449019:0.785398163397,variable=\t]({1.*0.131393206182*cos(\t r)+0.*0.131393206182*sin(\t r)},{0.*0.131393206182*cos(\t r)+1.*0.131393206182*sin(\t r)});
\draw [shift={(0.5,0.5)},dotted]  plot[domain=0.785398163397:3.92699081699,variable=\t]({1.*0.0853418478894*cos(\t r)+0.*0.0853418478894*sin(\t r)},{0.*0.0853418478894*cos(\t r)+1.*0.0853418478894*sin(\t r)});

\draw [dash pattern=on 2pt off 2pt] (0.,1.)-- (0.,-1.);
\draw [dash pattern=on 2pt off 2pt] (-1.,0.)-- (1.,0.);
\draw [dash pattern=on 1pt off 2pt on 3pt off 4pt] (0.5,1.)-- (0.5,0.);
\draw [dash pattern=on 1pt off 2pt on 3pt off 4pt] (0.,0.5)-- (1.,0.5);

\draw [shift={(0.000415320656828,0.000398940826174)}] plot[domain=-0.803722374116:2.33787027947,variable=\t]({1.*0.034275034219*cos(\t r)+0.*0.034275034219*sin(\t r)},{0.*0.034275034219*cos(\t r)+1.*0.034275034219*sin(\t r)});
\end{tikzpicture}
\caption{\label{sausage} A square $Q$ and its $4$ dyadic children $Q_i$ with the associated vector field $Z_Q$}
\end{figure}
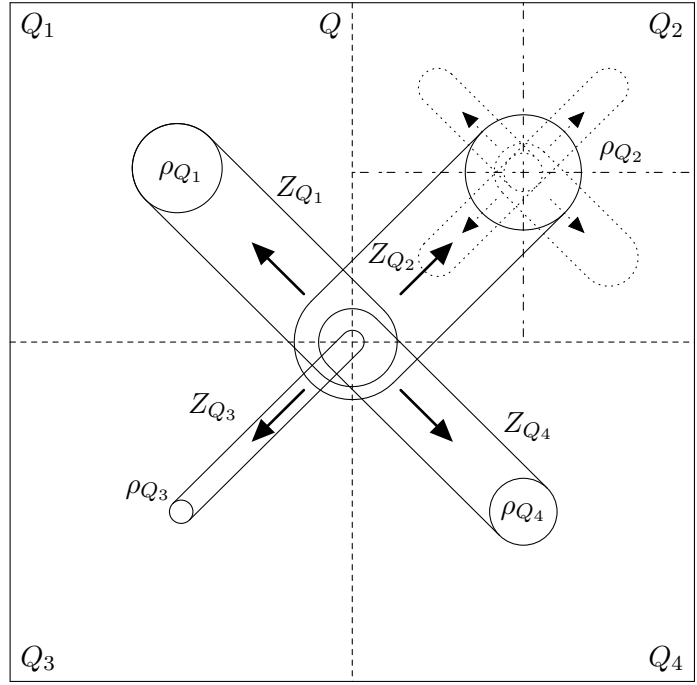
\noindent By definition of the kernel $\overline{\rho}_{Q}$, one has
\begin{equation}\label{malphazq}
M^{\alpha}_{\ve}(Z_Q)\leq C\theta_Q^{\alpha}\;\op{diam}(Q).
\end{equation}
This a consequence of the choice of $R_Q$ as a minimizer in $\eqref{heuristic}$. Indeed, for the sake of simplicity, let us assume that $\op{supp}\rho$ is the unit ball centered at the origin. Then $Z_Q$ is concentrated on a strip of width $R_Q$ around the segment $S=[c_{F(Q)},c_Q]$, i.e.  
\begin{equation}\label{strip}\op{supp}Z_Q\subset \tilde{S}:=\{x\in \R^d\;:\; \op{dist}(x,S)\leq R_Q\}\end{equation}
and $Z_Q$ satisfies the two estimates 
\begin{equation}\label{estimationzq}
\|Z_Q\|_{L^\infty}\leq C\theta_Q R_Q^{1-d}\quad\text{and}\quad\|\nabla Z_Q\|_{L^\infty}\leq C\theta_Q R_Q^{-d}.
\end{equation}
Then, the same computations as in $\eqref{heuristic}$ and the fact that $R_Q\leq \op{diam}(Q)$ give \eqref{malphazq}. 

Let us estimate the $L^1$ norm of $Y$ which has to be controlled by $\theta$ as stated in Theorem \ref{localestimate}:
\begin{equation*}\label{finalv} 
\|Y\|_{L^1(Q_0)}\leq \displaystyle\sum_{j\geq 1}\ \sum_{Q\in \mathcal{D}(f)\cap \mathcal{Q}_j}\|Z_Q\|_{L^1(\tilde{S})}\leq \displaystyle\sum_{j\geq 1}\ \sum_{Q\in \mathcal{D}(f)\cap \mathcal{Q}_j}\theta_Q\, L\,2^{-j}\ =L\,\theta.
\end{equation*}
Note that 
\begin{equation*}\label{div}
\nabla\cdot  Y=\rho_{Q_0}-h-\lambda_\ve  f,
\end{equation*}
where $h$ stands for the extra divergence. $h$ can be written as
$$h=\sum_{Q\in\mathcal{D}_{fr.}}\left\{\rho_Q-\sum_{Q'\in\mathcal{C}(Q)}\rho_{Q',Q}\right\},$$
where $\rho_{Q',Q}$ represents the kernel $\rho_{Q'}$ translated at $c_Q$, center of $Q$, and, for the sake of simplicity, the set of all cubes $Q$ such that $\mathcal{C}(Q)\subset\mathcal{D}(f)$ has been denoted by $\mathcal{D}_{fr.}$:
\begin{equation*}\label{D}
\mathcal{D}_{fr.}:=\{Q\in \mathcal{D}(f)\;:\; \mathcal{C}(Q)\subset \mathcal{D}(f)\}.
\end{equation*}
Since $\nabla\cdot  Y=\rho_{Q_0}-f +(f-\lambda_\ve  f)-h\neq\rho_{Q_0} -f$, we have to slightly modify the vector field $Y$. This will be done replacing $Y$ by 
$$V=Y+V_1+V_2,$$
where $V_1$, $V_2\in H^1(Q_0,\R^d)$ are constructed so that $\nabla\cdot  V_1= h$ and $\nabla\cdot V_2 = \lambda_\ve  f-f$. The construction of $V_1$ and the estimate of $M^\alpha_\ve (V_1)$, $\|V_1\|_{L^1}$ will be the object of the first step. In the second step we prove that $M^\alpha_\ve (Y)\leq C\theta^\alpha L$. Then, Proposition \ref{distdiffstate} allows us to construct $V_2\in H^1$ such that $\nabla\cdot  V_2=\lambda_\ve  f -f$ with an estimate on $M^\alpha_\ve (V_2)$ and $\|V_2\|_{L^1}$.
\paragraph{First step: Correction at the nodes, construction of $V_1$.}
For all $ Q\in\mathcal{D}_{fr.}$, let $B_Q$ be the support of $\rho_Q$. Since $\op{supp}\rho$ has been supposed to be the unit ball centered at the origin and $\rho_{Q}(x)=\theta_Q\rho_{R_Q}(x-c_Q)$, we have $B_Q=B(c_Q,R_{Q})\subset Q$. Let us define the extra divergence corresponding to this node,
$$
h_Q=\rho_{Q}-\displaystyle\sum_{Q'\in\mathcal{C}(Q)}\rho_{Q',Q}\quad .
$$
For each $Q\in\mathcal{D}_{fr.}$, thanks to Theorem \ref{dirichlet}, we can find $V_Q\in H_0^1(B_Q)$ such that $\nabla\cdot  V_Q=h_Q$ and $\|V_Q\|_{H^1(B_Q)}\leq C\;\|h_Q\|_{L^2(B_Q)}$. But in this case, because $h_Q$ is radial up to a translation, we essentialy use the proposition in dimension 1 which is quite easy and gives better estimates. Let us give more details on this point:
\begin{lem}\label{diffusionradial}
Let $d\geq 1$ and $B=B(0,R)\subset\R^d$ be a ball centered at the origin. There exists a constant $C>0$ only depending on $d$ such that the following holds:

Let $F\in L^\infty (B)$ be a radial function: i.e. for a.e. $x\in B$, $F(x)=f(|x|)$ for some $f\in L^\infty (0,R)$. Assume that $\int_B F=0$. Then, there exists a radial function $V\in W^{1,\infty}_{0}(B,\R^d)$ such that $\nabla\cdot V = F$ and
\begin{equation*}\label{bb}
\|\nabla V\|_{L^\infty (Q_0)}\leq C\;\|F\|_{L^\infty(Q_0)}\quad .
\end{equation*}
\end{lem}

\begin{proof}
First of all, by a scaling argument, one can assume that $R=1$. The vector field $V: B\to \R^d$ defined by $V(x)=v(|x|) x$ for some Lipschitz function $v: \R^+\to \R$ satisfies
$$\nabla \cdot V(x)= r^{1-d}[r^dv(r)]'$$
in the distributional sense. Thus, if $v$ is chosen as
$$v(r)=r^{-d}\int_0^r f(s)s^{d-1}\d s,$$
then $V$ solves the following problem:
$$
\begin{cases}
\nabla \cdot V(x) =F(x)& on $B$,\\
V(x) =0 & on $\partial B$.
\end{cases}
$$
Moreover, for a.e. $x\in B$, we have $\nabla V (x) = v'(|x|) x\otimes\frac{x}{|x|}+v(|x|) \op{Id}$, where $\op{Id}$ is the matrix identity. In particular, we get $\|\nabla V\|_{L^\infty} \leq C(\|rv'(r)\|_{L^\infty} +\|v\|_{L^\infty})$. The second term in the RHS on the preceding equation is estimated by $\|v\|_{L^\infty}\leq r^{1-d} \|f\|_{L^\infty}r^{d-1}=\|f\|_{L^\infty}$. For the first term, one has $v'(r)=-dr^{-d-1}\int_0^r f(s)s^{d-1}\d s+r^{-1}f(r)$ and so $\|rv'(r)\|_{L^\infty}\leq C\|f\|_{L^\infty}$. Thus, $\|\nabla V\|_{L^\infty}\leq C\|F\|_{L^\infty}$.
\end{proof}
Applying Lemma \ref{diffusionradial} to $F=h_Q$ and $R=R_{Q}$ gives $V_Q\in H^1_0 (B_Q)$ such that $\nabla \cdot V_Q=h_Q$ and
\begin{equation}\label{nablawq}
\|\nabla V_Q\|_{L^{\infty}(B_Q)}\leq C \theta_Q R_Q^{-d}\ ,\quad \|\nabla V_Q\|_{L^1(B_Q)}\leq |B_Q|\|\nabla V_Q\|_{L^\infty(B_Q)}\leq C\theta_Q.
\end{equation}
Moreover, since $V_Q$ is supported on $B_Q=B(c_Q,R_Q)$, we deduce that $\| V_Q\|_{L^{\infty}(B_Q)}\leq R_Q\|\nabla V_Q\|_{L^{\infty}(B_Q)}\leq C \theta_Q R_Q^{1-d}$ so that $V_Q$ satisfies the same estimate as \eqref{estimationzq}. In particular, we get $M^\alpha_\ve (V_Q)\leq C\theta_Q^\alpha\op{diam}(Q)$. Now define 
\begin{equation*}\label{W}V_1=\sum_{Q\in \mathcal{D}_{fr.}}V_Q.\end{equation*}
Since $\|V_Q\|_{L^1(B_Q)}\leq CR_Q\|\nabla V_Q\|_{L^1(B_Q)}\leq C\op{diam}(Q)\theta_Q$, Lemma \ref{malphafinite} implies 
$$\|V_1\|_{L^1(Q_0)}\leq C \op{diam}(Q_0) \theta_{Q_0}\leq C' L \|f\|_{L^1(Q_0)}$$ 
as required. Then, using the definition of $M^\alpha_\ve$ in \eqref{malphaepsilon} and the subadditivity of $x\to |x|^\beta$, one gets
\begin{equation}\label{malphaepsw1}
M^{\alpha}_{\ve}(V_1)\leq\ve^{-\gamma_1} \displaystyle\sum_{Q\in \mathcal{D}_{fr.}}\int|V_Q|^{\beta}+ 2\;\ve^{\gamma_2} \displaystyle\int\sum_{Q,Q'\in \mathcal{D}_{fr.}\;:\; Q'\subset Q}\left| \nabla V_{Q'}:\nabla V_Q\right|,
\end{equation}
where $A:B$ stands for the euclidian product of two matrices $A=(A_{ij})_{1\leq i,j\leq d}$, $B=(B_{ij})_{1\leq i,j\leq d}$ of size $d\times d$: $A:B:=\sum_{ij}A_{ij}B_{ij}$. For the estimate of $|\nabla V_1|^2$, we have used the identity $|\nabla V_1|^2=\nabla V_1 : \nabla V_1 = \sum_{Q,Q'\in\mathcal{D}_{fath}}\nabla V_Q :\nabla V_{Q'}$. Since $V_Q$ is supported on $Q$, $\nabla V_Q :\nabla V_{Q'}$ vanishes except when $Q\cap Q'\neq \emptyset$, i.e. $Q\subset Q'$ or $Q'\subset Q$, thus justifying the factor $2$ and the inclusion $Q'\subset Q$ in \eqref{malphaepsw1}. 

We need to estimate the two terms in \eqref{malphaepsw1}. Since $M^\alpha_\ve (V_Q)\leq C\theta_Q^\alpha\op{diam}(Q)$, thanks to Lemma \ref{malphafinite}, this term is less or equal than $C \theta^\alpha L$ as required. Using the inequality $\|fg\|_{L^1}\leq \|f\|_{L^\infty}\|g\|_{L^1}$, one can estimate the second term of \eqref{malphaepsw1} by
$$2\;\ve^{\gamma_2} \displaystyle\sum_{Q,Q'\in \mathcal{D}_{fr.}\;:\; Q'\subset Q}\|\nabla V_{Q}\|_{L^{\infty}(B_Q)}\|\nabla V_{B_{Q'}}\|_{L^1(B_{Q'})}.
$$
Note that it would be more natural to use a Cauchy-Schwarz inequality ($L^2$-$L^2$) at this step but, using it, we were not able to deduce the estimate by $\theta^\alpha L$. Once again, since $R_{Q'}\leq \op{diam}(Q')$, we have
\begin{equation}\label{thesame}
\|\nabla V_{Q'}\|_{L^1(B_{Q'})}\leq C \theta_{Q'}\leq \op{diam}(Q')R_{Q'}^{-1}\theta_{Q'}=C\op{diam}(Q')\ve^{-\gamma}\theta_{Q'}^{1-\frac{1-\gamma}{d-1}}.
\end{equation}
Since $1-\frac{1}{d}<1-\frac{1-\gamma}{d-1}<1$, Lemma \ref{malphafinite} gives
$$\sum_{Q'\in\mathcal{D}_{fr.}\;:\; Q'\subset Q}\|\nabla V_{Q'}\|_{L^1(B_{Q'})}\leq C\ve^{-\gamma}\op{diam}(Q)\theta_Q^{1-\frac{1-\gamma}{d-1}}.$$
Now, elementary computations on exponents $\alpha$, $\gamma_2$, $\gamma$ and Lemma \ref{malphafinite} give successively $\gamma_2=(d+1) \gamma$,  $\alpha=2-(d+1)\frac{1-\gamma}{d-1}$ and
$$C\ve^{\gamma_2}\displaystyle\sum_{Q\in \mathcal{D}_{fr.}} \op{diam}(Q)\theta_Q R_Q^{-d}\ve^{-\gamma}\theta_Q^{1-\frac{1-\gamma}{d-1}}=C\displaystyle\sum_{Q\in \mathcal{D}_{fr.}}\op{diam}(Q)\theta_Q^{\alpha} \leq C \theta^{\alpha} L.$$
Finally, we have obtained the desired inequality: $M^{\alpha}_{\ve}(V_1)\leq C\; \theta^{\alpha} L$.

\paragraph{Second step: estimate of the energy of $Y$ on the node set.}
In order to get estimates on $Y$, it is convenient to divide $Q_0$ into 2 domains: the node set $N$ and its complementary $N^c$, where
\begin{equation*}\label{node}
N:=\displaystyle\bigcup_{Q\in\mathcal{D}(f)}B(c_Q, cR_Q)
\end{equation*}
and $c>0$ is a constant which will be chosen later. By analogy with $V_1$, one can write $Y_{|N}$ as a sum of vector fields $Y_Q$, where
$$Y_Q=\begin{cases}
\ind_{B(c_Q, cR_Q)}\left(Z_Q-\sum_{Q'\in \mathcal{C}(Q)} Z_{Q'}\right)&if $Q\in\mathcal{D}_{fr.}$ (see \eqref{uq}),\\
\ind_{B(c_Q, cR_Q)}Z_Q&otherwise.
\end{cases}
$$
Now, from \eqref{estimationzq}, we deduce the estimates \eqref{nablawq} satisfied by $V_Q$ are also true for $Y_Q$ and consequently, we obtain $M^{\alpha}_{\ve}(Y,N)\leq C\; \theta^{\alpha} L$ as well (see \eqref{malphaepsilon} for the definition of $M^{\alpha}_{\ve}(Y,N)$).

\paragraph{Third step: estimate of the energy of $Y$ out of the node set.}
Reminding that
$$Y=\displaystyle\sum_{\substack{Q\in\mathcal{D}(f)\\Q\neq Q_0}}Z_Q\quad ,$$
considering that $M^\alpha_\ve$ is not subadditive (due to the term $|\nabla Y|^2$), the first thing to do is to understand to which extent the supports of $Z_Q$ can intersect. To this aim, let us note that if the constant $c>0$ in \eqref{node} is chosen equal to $\sqrt{d}$ or more, due to \eqref{strip}, then each $Z_{Q}$ restricted to $N^c$ is supported on $Q$ (see figure \ref{sausage}): $\op{supp}Z_Q\cap N^c\subset Q$. In particular, this implies that
\begin{equation*}\label{qq'}
\op{supp}Z_Q\cap\op{supp}Z_{Q'}\cap N^c\neq \emptyset \ \Longrightarrow \  Q\cap Q'\neq \emptyset \ \Longrightarrow\ Q\subset Q'\quad\text{or}\quad Q'\subset Q.
\end{equation*}
For this reason, $M^\alpha_\ve (Y,N^c)$ can be estimated exactly in the same way as we did for the estimate of $M^\alpha_\ve (V_1)$ in \eqref{malphaepsw1}. Moreover, the Young inequality, $\|f\ast \mu\|_{L^1}\leq \|f\|_{L^1} |\mu| (\R^d)$, valid for all $f\in L^1(\R^d)$, $\mu\in\mathcal{M}(\R^d)$, and the definition of $Z_Q$ in \eqref{uq}, easily give
$$\|\nabla Z_{Q'}\|_{L^1(Q')}\leq C\theta_{Q'}R_{Q'}^{-1} \op{diam}(Q').$$
Since this estimate (which is the same as \eqref{thesame}) and \eqref{malphazq} are the only ones we have used in the first step for the estimate of $M^\alpha_\ve (V_1)$, we get $M^{\alpha}_{\ve}(Y,N^c)\leq C\theta^\alpha L$ as well.

\paragraph{End of the proof of Theorem \ref{localestimate}}
Finally, the vector field $V=Y+V_1+V_2$, where $V_2$ is given by Proposition \ref{distdiffstate}, satisfies $\nabla\cdot V=\rho_{Q_0} -f$,
$$M^\alpha_\ve (V)\leq 3\{ M^\alpha_\ve (Y)+M^\alpha_\ve (V_1)+M^\alpha_\ve (V_2)\}\leq C\{\theta^\alpha L+\ve^{\gamma_2}\|f\|_{L^2}^2\}$$
and 
$$\|V\|_{L^1}\leq \|Y\|_{L^1}+\|V_1\|_{L^1}+\|V_2\|_{L^1}\leq CL \|f\|_{L^1}.$$

\section{Estimate between $d^{\alpha}_{\ve}$ and the Wasserstein distance\label{comparisonsection}}

Our aim is to prove an estimate on the pseudo-distances $d^\alpha_\ve$ similar to Proposition \ref{comparisonprop}. Because of the Dirichlet term in the definition of $M^\alpha_\ve$, $d^\alpha_\ve$ cannot be estimated only by the Wasserstein distance $W_1$ but one has to add a term involving $\|f^+-f^-\|_{L^2}$. Using Theorem \ref{localestimate}, we are going to prove the following theorem:
\begin{thm}\label{comparisonwithwasserstein}
Let $Q=(0,L)^d$ be a a cube of side length $L>0$ in $\R^d$ and $\ve\in (0,1)$. There exists $C>0$ only depending on $\alpha$, $d$ and $L$ such that for all $f^+, f^-\in L_+^2(Q)$ with $\int_{Q}f^+=\int_{Q}f^-=1$, there exists $u\in H^1(\R^d)$ compactly supported on the set $Q_\ve:=\{x\in\R^d\;:\; \op{dist}(x,Q)\leq C\ve^\gamma\}$ satisfying $\nabla\cdot  u=f:=f^+-f^-$ as well as
\begin{equation}\label{globalest}d^\alpha_\ve(f^+,f^-)\leq M^{\alpha}_{\ve}(u)\leq C\, H\big( W_{1}^{1-d(1-\alpha)}(f^+,f^-)+\ve^{\gamma_2} 
\|f\|_{L^2}^2\big)\quad\text{and}\quad  \|u\|_{L^1}\leq C,\end{equation}
where $H:\R^+\longrightarrow \R^+$ is the scalar function defined by $H(x)=x+x^\lambda$ for some $\lambda\in (0,1)$ depending on $\alpha$, and $W_{1}$ stands for the Wasserstein distance associated to the Monge cost $(x,y)\to |x-y|$.
\end{thm}
\begin{rmk}\label{dependancetheta}
One can replace the condition $\int f^\pm=1$ by $\int f^\pm=\theta\geq 0$. Then, the constant $C$ will also depend on $\theta$: $C=C(\theta,\alpha,d,L)$. However, we can easily check that $C$ is locally bounded with respect to $\theta$, i.e. it is uniform for bounded values of $\theta$.
\end{rmk}

\begin{rmk}
It is tempting to think that estimate \eqref{globalest} also holds when $H(x)=x$ which would be the natural choice. Indeed, if $\ve$ is taken very small, since $M^\alpha_\ve$ $\Gamma$-converge to $M^\alpha$ and because of Proposition \ref{comparisonprop}, one can expect that $d^\alpha_\ve(f^+,f^-)\simeq d^\alpha(f^+,f^-)\leq CW_{1}(f^+,f^-)^{1-d(1-\alpha)}$. On the contrary, when $\ve$ is very large, because of Theorem \ref{dirichlet}, one can expect that $d^\alpha_\ve(f^+,f^-)\simeq \ve^{\gamma_2} \|f\|_{L^2}^2$. However, for technical reasons, due to the lack of subadditivity of the second term (Dirichlet energy) in the definition of $M^\alpha_\ve$, we were not able to reach the case $H(x)=x$.
\end{rmk}
\begin{proof}
Our method to prove this proposition is an adaptation of that of J.-M. Morel and F. Santambrogio in \cite{Morel:2007} (see also Proposition 6.16. page 64 in \cite{Bernot:2009}).

Up to replacing $(f^+,f^-)$ by $(f^+-f^+ \wedge f^-,f^--f^+\wedge f^-)$, one can assume that $f^+\wedge f^- =0$, where for all $x\in Q$, $(f^-\wedge f^+)(x)=\inf(f^-(x),f^+(x))$. Indeed, it is sufficient to note that, if $\mu^\pm$ are two measures with the same mass and $\nu$ is a positive measure on $Q$ then we have $W_1(\mu^++\nu,\mu^-+\nu)= W_1(\mu^+,\mu^-)$.

For the sake of simplicity, in all the proof, $C>0$ will denote some constant only depending on $\alpha$, $d$ and $L$ and big enough so that all the inequalities below are satisfied.

Let $f^+,f^-\in L^2_+(Q)$ be two densities on the cube $Q=(0,L)^d$ such that $\int_Q f^\pm=1$. Chose an optimal transport plan $\Pi$ between $f^+$ and $f^-$ for the Monge-Kantorovich problem associated to the cost $c(x,y)=|x-y|$. Hence $\Pi$ satisfies the constraint $P_\#^\pm\Pi=f^\pm (x)\d x$ where $P^+$ (resp. $P^-$) is the projection on the first variable $x$ (resp. the second variable $y$) and $\d x$ is the Lebesgue measure. Moreover we have
\begin{equation}\int_{Q}|x-y| \d\Pi(x,y)=W_1(f^+,f^-)=:W\label{Wass}.\end{equation}
So as to use the local estimate of the previous part, let us classify the set of ordered pairs $(x,y)$ with respect to the distance $|x-y|$. More precisely, for $j\geq 0$, set
$$X_j=\{(x,y)\in Q^2\;:\; d_j\leq |x-y| < d_{j+1}\},
$$
where $d_j=(2^j-1)\, w$ and $w\in (0,1)$ will be chosen later. In particular, $d_0=0$ and $X_j$ is empty if $d_j>\op{diam}(Q)$, i.e. $j> J:=\left\lfloor \ln_2\left(\frac{\op{diam}(Q)}{w}+1\right)\right\rfloor$. For this reason, one can restrict to integers $j\leq J\leq C(1+|\ln w|)$: we will assume that $d_j\leq \op{diam}(Q)$. Moreover, \eqref{Wass} immediately gives the estimate
\begin{equation}\label{estwp}
\sum_j d_j \theta_j\leq W\, ,\hspace{3pt}\quad\text{where}\quad\hspace{5pt}\theta_j=\Pi (X_j).
\end{equation}
Next, for each integer $j\in [1,J]$, consider a uniform partition of $Q$ into cubes $Q_{jk}$, $k=1,\dots, K_j$, with side length $d_{j+1}$.
It is easy to estimate $K_j$ by
\begin{equation}\label{estkj}K_j\leq C d_{j+1}^{-d}.\end{equation}
For $j\geq 0$, set 
$$\Pi_j=\Pi_{|X_j}\ ;\ \theta_j=\Pi(X_j);\ f^\pm_j =P_\#^\pm\Pi_j \quad\text{and}\quad f_j=f_j^+-f_j^-,$$
Clearly, one has
\begin{equation*}\label{sumfj}
\Pi=\sum_j \Pi_j\quad\text{and}\quad f^\pm =\sum_j f^\pm_j.
\end{equation*}
In the same way, for $j\geq 0$ and $1\leq k\leq K_j$, set
$$\Pi_{jk}=\Pi_{|X_j\,\cap\,  (Q_{jk}\times Q)}\ ;\ \theta_{jk}=\Pi_{jk}(Q^2)\quad\text{and}\quad f_{jk}^\pm=P^\pm_\#\; \Pi_{jk}$$
so that
$$\Pi_{j}=\sum_k \Pi_{jk};\ \theta_j=\sum_k \theta_{jk}\quad\text{and}\quad f^\pm_j =\sum_k f_{jk}^\pm.$$
$\Pi_{jk}$ represents the part of the transport plan $\Pi$ corresponding to points in $Q_{jk}$ which are sent at a distance comparable to $d_{j+1}$. In particular, $f^+_{jk}$ is supported on $Q_{jk}$ and $f^-_{jk}$ is supported on the cube $\widetilde{Q}_{jk}$ with the same center but twice the side length of $Q_{jk}$. As we did in \eqref{defrhoq}, let us define $\rho_{jk}$ the kernel associated to $Q_{jk}$ by
$$\rho_{jk}(x)=
\begin{cases}
(R_{jk})^{-d}\rho (R_{jk} (x-c_{jk}))&if $\theta_{jk}\neq 0$,\\
0&otherwise,
\end{cases}$$
where $\rho\in\c^1_c (\R^d,\R^+)$, $R_{jk} = \ve^\gamma \theta_{jk}^\frac{1-\gamma}{d-1}$ and $c_{jk}$ is the center of $Q_{jk}$. For the sake of simpli\-city, let us assume that $\op{supp}\rho$ is the unit ball centered at the origin. Let $B_{jk}:=B(c_{jk},r_{jk})$ be the smallest ball containing $\widetilde{Q}_{jk}$ and $\op{supp}\rho_{jk}=B(c_{jk},R_{jk})$: i.e. $r_{jk}=\op{max}\{R_{jk}, \op{diam}(Q_{jk})\}$. Thanks to Theorem \ref{localestimate}, it is possible to find a vector field $u_{jk}\in H^1_0(B_{jk})$ satisfying $\nabla\cdot  u_{jk}=f_{jk} :=f^+_{jk}-f^-_{jk}$\,,\; $\|u_{jk}\|_{L^1(B_{jk})}\leq C \theta_{jk}$ and 
\begin{equation}\label{estloc}M_{jk}:=M^\alpha_\ve(u_{jk})\leq C\, \{\theta_{jk}^{\alpha}d_{j+1} + \ve^{\gamma_2} \|f_{jk}\|^2_{L^2(B_{jk})}\}.
\end{equation}
Moreover, if $R_{jk}\geq d_{j+1}/2$, the first term in the right-hand side of \eqref{estloc} can be omitted since one has \begin{equation}\label{onlyoneterm}
\theta_{jk}^\alpha d_{j+1}\leq C\ve^{\gamma_2}\|f_{jk}\|_{L^2}^2.
\end{equation}
Indeed, in this case, writing $\theta:=\theta_{jk}$ and $R:=R_{jk}$, one has $\theta^\alpha d_{j+1}\leq 2\theta^\alpha R$ and, using $2-\alpha=\frac{(1-\gamma)(d+1)}{d-1}$, we get $\theta^\alpha R=[\theta^{\alpha-2}R^{1+d}][\theta^2 R^{-d}]=\ve^{\gamma_2} R^{-d}\theta^2$. Then, \eqref{onlyoneterm} follows from the fact that, by the Cauchy-Schwarz inequality, we have 
$$R^{-d} \theta^2\leq R^{-d}|B_{jk}|\int_{B_{jk}}(f_{jk})^2\leq C\int_{B_{jk}} (f_{jk})^2.$$
Now, let us define the vector field $u=\displaystyle\sum_{j,k} u_{jk}$, which satisfies
$$\nabla\cdot  u=\displaystyle\sum_{j,k} \nabla\cdot  u_{jk}=\displaystyle\sum_{j,k} f_{jk} =f:=f^+-f^- .$$
First note that
$$
\|u\|_{L^1(Q)}\leq  C\sum \|u_{jk}\|_{L^1(B_{jk})}
\leq 2C\sum\theta_{jk}=2C.$$
In order estimate the energy of $u$, a similar development of $\left\vert \sum \nabla u_{jk}\right\vert^2$ as in \eqref{malphaepsw1} and the Cauchy-Schwarz inequality give
\begin{equation}\label{est1}
 M^\alpha_\ve(u) \leq J \displaystyle\sum_{j=1}^J M^\alpha_\ve  \left(\displaystyle\sum_{k=1}^{K_j} u_{jk}\right)\leq C\; J\displaystyle\sum_j\left\{ \displaystyle\sum_{k}M_{jk}
+\displaystyle\sum_{(k,l)\in I_j} \sqrt{M_{jk}}\sqrt{M_{jl}}\right\},
\end{equation}
where $I_j$ stands for the set of pairs $(k,l)$ satisfying $k\neq l$, $\theta_{jk}\geq\theta_{jl}$ and $B_{jk}\cap B_{jl} \neq \emptyset$. We have to estimate the two terms in the right-hand side of \eqref{est1}.

\paragraph{Estimate of the first term in \eqref{est1}} We recall that $M_{jk} \leq \theta_{jk}^\alpha d_{j+1} +\ve^{\gamma_2} \|f_{jk}\|_{L^2(B_{jk})}^2$. For the second term, note that
\begin{equation}\label{antisubadditive}\textstyle\sum_{j,k} \|f_{jk}\|^2_{L^2}\leq \|f\|^2_{L^2}.\end{equation}
Indeed, since $f^+\wedge f^- =0$, for all $j,k$, one has $f_{jk}^+\wedge f_{jk}^- =0$ as well. In particular, $\|f_{jk}\|^2_{L^2(B_{jk})}= \|f_{jk}^+\|^2_{L^2(B_{jk})}+ \|f_{jk}^-\|^2_{L^2(B_{jk})}$\,,\; $\|f\|^2_{L^2(Q)}= \|f^+\|^2_{L^2(Q)}+ \|f^-\|^2_{L^2(Q)}$ and \eqref{antisubadditive} follows from the super-additivity of the power function $x\to |x|^p$ for $p\geq 1$: $|x+y|^p\geq |x|^p+|y|^p$ for $x,y\in\R$ whenever $xy\geq 0$.

For the first term, applying successively the Jensen inequality with power $\alpha\in (0,1)$, the H\"older inequality, \eqref{estwp} and the fact that $K_jd_{j+1}=C d_{j+1}^{1-d}$ (see \eqref{estkj}), one gets
\begin{align*}
\sum_{j,k}\theta_{jk}^\alpha d_{j+1}&\leq\sum_j d_{j+1} K_j [\theta_j/K_j]^\alpha=\sum_j [d_{j+1}\theta_j ]^\alpha [d_{j+1}K_j]^{1-\alpha}\\
&\leq \left(\sum_j \theta_{j}\,d_{j+1}\right)^\alpha  \left(\sum_{j} d_{j+1} K_j\right)^{1-\alpha}\\
&\leq C(w+W)^\alpha  \left(\sum_{j} [w(2^{j+1}-1)]^{1-d}\right)^{1-\alpha}\\
&\leq C' (w^\alpha+W^\alpha) w^{(1-d)(1-\alpha)}
\end{align*}
since $\theta_0 d_1\leq d_1=w$ (we cannot estimate this term by $W$ because $d_0=0$) and, because of \eqref{estwp}, $\sum_{j\geq1} \theta_{j}\,d_{j+1}\leq 3\sum_{j\geq 1} \theta_{j}\,d_{j} \leq 3W$. Finally, we get
\begin{equation}\label{estn1}
\textstyle\sum_{j,k}M_{jk} \leq C\left\{ w^{1-d(1-\alpha)}+W^\alpha w^{-(d-1)(1-\alpha)}+\ve^{\gamma_2} \|f\|_{L^2}^2\right\}.
\end{equation}

\paragraph{Estimate of the second term in \eqref{est1}}

Before following these computations, we need to understand what the condition ``$B_{jk}\cap B_{jl} \neq \emptyset\,$'' is meaning. Assume that $(k,l)\in I_j$. From $Q_{jk}\cap Q_{jl} = \emptyset $, we see that either $\op{supp}\rho_{jk}$ or $\op{supp}\rho^j_l$ is not included in $Q_{jk}$ (resp. $Q_{jl}$). Since, by definition of $I_j$, we have $\theta_{jk}\geq \theta_{jl}$, this implies that $R_{jk}\geq d_{j+1}/2$. Therefore, as we noticed after formula \eqref{estloc},
$$M_{jk}\leq \ve^{\gamma_2}\|f_{jk}\|_{L^2(B_{jk})}^2$$
and \eqref{onlyoneterm} also implies that
$$ \theta_{jl}^\alpha d_{j+1}\leq\theta_{jk}^\alpha d_{j+1} \leq C\ve^{\gamma_2}\|f_{jk}\|_{L^2(B_{jk})}^2.$$
Now, \eqref{estloc}, the subadditivity of the square root function, the preceding inequality, \eqref{antisubadditive} and Cauchy-Schwarz inequality give in turn
\begin{align*}
\dsum_{(k,l)\in I_j} \sqrt{M_{jk}}\sqrt{M_{jl}}&\leq  C\dsum_{(k,l)\in I_j} \sqrt{\ve^{\gamma_2}\|f_{jk}\|_2^2}\left(\sqrt{\ve^{\gamma_2}\|f_{jl}\|_{2}^2} 
+\sqrt{\theta_{jl}^\alpha d_{j+1} } \right)\\
&\leq  C\ve^{\gamma_2}\dsum_{(k,l)\in I_j}\|f_{jk}\|^2_{2}+\|f_{jk}\|_{2}\|f_{jl}\|_{2}\\
&\leq  C\ve^{\gamma_2} \left\{K_j\|f_j\|_{L^2(Q)}^2 +\sqrt{\sum_{k,l}\|f_{jk}\|_{2}^2} \sqrt{ \sum_{k,l}\|f_{jk}\|_{2}^2}\right\}\\
&\leq 2C\ve^{\gamma_2} K_j \|f_j\|^2_{L^2(Q)}.
\end{align*}
From $K_j\leq d_{j+1}^{-d}\leq 2^{-dj}w^{-d}$ and $\|f_j\|^2_{L^2(Q)}\leq \|f\|^2_{L^2(Q)}$, we obtain in the end that
\begin{equation}\label{estn2}
\dsum_j\displaystyle\sum_{(k,l)\in I_j} \sqrt{M_{jk}}\sqrt{M_{jl}}\leq Cw^{-d} \ve^{\gamma_2} \|f\|^2_{L^2}.
\end{equation}

\paragraph{End of the proof}
Let $F=\ve^{\gamma_2}\|f\|^2_{L^2}$. We remind the definition of $W=W_1(f^+,f^-)$. One can assume that $f^-\neq f^+$ so that $F,W>0$. Now, \eqref{est1}, \eqref{estn1}, \eqref{estn2} and the fact that $J\leq C (1+\ln w)$ yield
$$
M^\alpha_\ve  (u)\leq C(1+|\ln w|)\left\{ w^{\nu}+W^\alpha w^{\nu-\alpha}+w^{-d}F\right\},
$$
where $\nu :=1-d(1-\alpha)\in(0,1)$ and so $\alpha-\nu=-(d-1)(1-\alpha)<0$. Let us fix some $\delta\in (0,1)$ small enough so that $0<\nu\pm\delta<1$ and $\nu-\alpha\pm\delta<0$. For some constant $c$ depending on $\delta$, one has $1+|\ln w|\leq c(w^\delta+w^{-\delta})$ and so
$$M^\alpha_\ve  (u)\leq C\left\{ w^{\nu\pm\delta}+W^\alpha w^{\nu-\alpha\pm\delta}+w^{-d\pm\delta}F\right\},
$$
where the sum is taken over the values of $\pm 1$ ($+1$ or $-1$) in the right-hand side. Then, we make the choice $w=W+ F^\lambda >0$ for some $\lambda=\lambda(\alpha,d)>0$ which will be fixed later. Note that all the estimates above are valid only if $w<1$. However, if $W+F^\lambda\geq 1$ then the right-hand side of \eqref{globalest} is greater than some positive constant and \eqref{globalest} easily follows from Theorem \ref{localestimate} since $H(x)\geq x$. Thus, one can assume that $w\in (0,1)$.

Since $0<\nu\pm\delta<1$, we get $w^{\nu\pm \delta}\leq W^{\nu\pm \delta}+ F^{\lambda(\nu\pm \delta)}$ and, because $-d\pm\delta <0$, $\nu-\alpha\pm\delta<0$, we have $w^{\nu-\alpha\pm\delta}\leq W^{\nu-\alpha\pm\delta}$ and $w^{-d\pm\delta}\leq F^{\lambda (-d\pm\delta)}$ which gives
$$
M^\alpha_\ve  (u)\leq C\left\{ W^{\nu\pm\delta}+F^{\lambda (\nu\pm\delta )}+ W^{\nu\pm\delta}+F^{1+\lambda (-d\pm\delta)}\right\}.
$$
We fix $\lambda>0$ small enough so that $1+\lambda (-d\pm\delta)>0$: in this way, all the exponents in the preceding formula are positive. Finally, \eqref{globalest} follows from the fact that we have $W$, $F\leq 1$ as a consequence of $W$, $F\leq W^{1-d (1-\alpha)}+F$.
\end{proof}
\begin{rmk}
Since $\min\{w^\nu +w^{-d} F\; :\; w\in(0,1)\}=c F^{\frac{1}{d+\nu}}$ and $\frac{1}{d+\nu}<1$, one cannot obtain an estimate of the form $M^\alpha_\ve (u)\leq C(W+F)$ as expected. However, one could improve a bit \eqref{globalest} by a better estimate of the number of indices $l$ such that $(k,l)\in I_{j}$.
\end{rmk}

\section{A $\Gamma$-convergence result\label{gammacvsection}}
Let $\Omega\subset \R^2$ be a bounded open set and $\mu=\mu^+-\mu^-$ be a finite measure, where $\mu^\pm$ are two probability measures compactly supported on $\Omega$. We recall the definition of the set
$$\mathcal{M}_{div}(\Omega)=\{u :\Omega\to \R^2\,:\, u\text{ and }\nabla\cdot  u\text{ are finite measures on } \overline{\Omega}\}$$
which is endowed with the topology of weak star convergence on vector measures and their divergence. As weak star topology is never metrizable in infinite dimensional Banach spaces, the space $\mathcal{M}_{div}(\Omega)$ is not metrizable. Indeed, assume that $X$ is some infinite dimensional Banach space such that $X'$ is metrizable. In particular $X'$ admits a countable neighborhood basis $(V_n)_{n\geq 1}$ which one can assume to be of the form
$$V_n=\{\varphi\;:\; |\langle\varphi\;;x_i\rangle|<\varepsilon_n\text{ for }i=1,\dots,n\}$$
for some linearly independent family of vectors $(x_i)_{i\geq 1}\subset X$ and $\varepsilon_n>0$. Then the Hahn-Banach Theorem easily provides a sequence $(\varphi_n)_{n\geq 1}$ satisfying $\varphi_n(x_i)=0$ for all $i\leq n\in\mathbb{N}^*$ and $\|\varphi_n\|_{X'}=n$. In particular the sequence $(\varphi_n)_n$ weakly converges to $0$ as $n\to\infty$ which is a contradiction with the fact that $(\varphi_n)_n$ is norm unbounded. 

However, every bounded subsets of the dual space of a separable Banach space are metrizable for the weak star topology. In particular, for the natural norm $\|u\|_{\mathcal{M}_{div}(\Omega)}=|\nabla\cdot  u|(\Omega)+|u|(\Omega)$ given by the total variation of $u$ and its divergence, we know that all bounded subsets of $\mathcal{M}_{div}(\Omega)$ are metrizable: for all $M>0$, there exists a metric $d_M$ for the weak star convergence of $u$ and $\nabla\cdot  u$ on the set
$$\mathcal{M}_M(\Omega)=\{u\in \mathcal{M}_{div}(\Omega)\;:\; |u|(\Omega)+|\nabla\cdot  u|(\Omega)\leq M\}.$$
In \cite{Oudet:2011} the $\Gamma$-convergence of the functional sequence $M^\alpha_\ve$ to $M^\alpha$ was proved. Our aim is to prove that this property remains true when adding a divergence constraint. Since, for $u\in H^1(\Omega)$, one has $\nabla\cdot u\in L^2$, one cannot prescribe $\nabla\cdot u=\mu$ if $\mu$ is not in $L^2$. For this reason, we first have to define a regularization of $\mu$. Let $(f_\ve)_{\ve>0}\subset L^2$ be a sequence of $L^2$ functions weakly converging to $\mu$ as measures and satisfying
\begin{equation}\label{hypponfeps}
\int_\Omega f_\varepsilon(x)\d x=0\quad\text{and}\quad\ve^{\gamma_2}\|f_\ve\|_{L^2}^2\underset{\ve\to 0}{\longrightarrow} 0\quad .
\end{equation}
This choice is going to be useful for the proof of Theorem \ref{gammacv}. For example, we can define $f_\ve$ as
\begin{equation*}\label{fepsilon}f_\ve:=\rho_\ve\ast \mu ,\end{equation*}
where $\rho_\ve (x)=\ve^{-2\gamma}\rho(\ve^{-\gamma}x)$ for some compactly supported $\rho\in\c^1 (\R^d,\R^+)$ such that $\int_\Omega \rho =1$ and $\gamma$ is still defined as $\gamma=\frac{\gamma_2}{d+1}=\frac{\alpha+1}{3}$. Now, let us define the functionals $\overline{M}^\alpha_{\ve}$ (resp $\overline{M}^\alpha$) adding a divergence constraint on $u\in\mathcal{M}_{div}(\Omega)$:
\begin{equation*}
\overline{M}^\alpha (u) = 
\begin{cases}
M^\alpha (u) & if $\nabla\cdot  u=\mu ,$ 
\\ +\infty  &otherwise,
\end{cases}
\end{equation*}
\begin{equation*}\label{mbar}
\overline{M}^\alpha_{\ve} (u) = 
\begin{cases}
M^\alpha_\ve (u) & if $\nabla\cdot  u=f_\ve$, 
\\ +\infty  &otherwise.
\end{cases}
\end{equation*}
The main result of this section is Theorem \ref{gammacv}:
\begin{thm*}\label{gammacvlast}
There exists a constant $c_0$ such that the functional sequence $(\overline{M}^\alpha_{\ve})_{\ve>0}$ $\Gamma$-converges to $c_0\overline{M}^\alpha$ as $\ve\to 0$. Moreover $c_0$ is given by the minimum value for the minimization problem $\eqref{hilliard}$.
\end{thm*}
We first remind how to build a recovery sequence in the case of a mass $\theta$ flowing on a single segment $S$, i.e. $u=\theta\hone_{|S}$. To this aim, we need to find a structure close to $u$ which is almost optimal for $M^\alpha_\ve$. We proceed by a slicing argument:

Let $u$ be any vector measure in $\mathcal{M}_{div}(\Omega)$. Take some $\nu\in S^{1}:=\{x\in\R^2\;:\; |x|=1\}$ which has to be thought as the tangent vector to $S$ in the case where $u=\theta\hone_{|S}$. Let us consider $v=[(u\cdot \nu)_+]_{|\nu^\perp}$ (restriction on $\nu^\perp$ of the positive part of $u\cdot \nu$) the flux of $u$ across the hyperplane $\nu^\perp =\{x\in\R^2\;:\; x\cdot \nu=0\}$ and assume that $\int v=\theta$. Then $M^\alpha_\ve(u)$ can be controlled from below by integrals on subintervals of $\R\nu$ of the following Cahn-Hilliard type energy (see \cite{Cahn:1958} for physical motivations):
$$F^\beta_\ve (v)=\ve^{-\gamma_1}\int_\R v^\beta+\ve^{\gamma_2}\int_\R |\nabla v|^2.$$
This kind of models for droplets equilibrium was studied by G. Bouchitt\'e, C. Dubs and P. Seppecher in \cite{Bouchitte:1996} for instance (see also \cite{Bouchitte:1994}). $F^\beta_\ve (v)$ can be renormalized through the formula: $v(x)=\theta R_{\theta,\ve}^{-d} w(R_{\theta,\ve}^{-1} x)$, where $R_{\theta,\ve}=\ve^\gamma\theta^{\frac{1-\gamma}{d-1}}$. Then, the constraint $\int v=\theta$ turns into $\int w=1$ and $F^\beta_\ve (v)=\theta^\alpha F^\beta (w)$, where
$$F^\beta (w)=\left\{ \int_\R w^\beta+\int_\R |\nabla w|^2 \right\}.$$
Then, the existence of an optimal profile $w$ is given by
\begin{lem}\label{chan}There exists a profile $w\in H^1_{loc}(\R,\R^+)$ solution of the minimization problem 
\begin{equation}\label{hilliard}\min\left\{\int_\R w^\beta+\int_\R |w'|^2\;:\; w\in H^1_{loc} (\R,\R^+)\quad\text{and}\quad \int_{\R}w=1\right\}.\end{equation}
Moreover, $w$ is compactly supported, Lipschitz continuous on $\R$ and $\,\mathcal{C}^\infty$ inside its support, i.e. on the open set $\{w>0\}$. Last of all, it is possible to choose $w$ such that it is even and non-increasing on $\R^+$.
\end{lem}
\begin{rmk}Although we restrict to the two dimensional case, some works by G. Bouchitt\'e, C. Dubs and P. Seppecher (see \cite{Dubs:1998}) suggest that Lemma \ref{chan} and Theorem \ref{gammacv} could be generalized in every dimension as well. However, the aim of this paper is to use the tools of section \ref{comparisonsection} so as to establish the $\Gamma -\limsup$ property for functionals $\overline{M}^\alpha_\ve$ (with divergence constraint) and, from the point of view of the complexity of the proof, this is independent of the dimension. Since the $\Gamma -\liminf$ property was only established in the $2D$ case in \cite{Oudet:2011}, we prefer to stay in this framework. Actually, the difficulty to prove a $\Gamma$-convergence result of $\overline{M}^\alpha_\ve$ (resp. $M^\alpha_\ve$) to $\overline{M}^\alpha$ (resp. $M^\alpha$) in every dimension would concern the $\Gamma -\liminf$ part and this is not the purpose of this paper.
\end{rmk}
\begin{rmk}Note that the minimum value in \eqref{hilliard} is related to the best constant in the one-dimensional Gagliardo-Nirenberg inequality $\int_\R |u|\leq C \left(\int_\R |u'|^2\right)^{\frac{1-\beta}{2+\beta}}\left(\int_\R |u|^\beta\right)^{\frac{3}{2+\beta}}$:
$$\frac{1}{C}=\inf\left\{\left(\displaystyle\int_\R |u'|^2\right)^{\frac{1-\beta}{2+\beta}}\left(\displaystyle\int_\R u^\beta\right)^{\frac{3}{2+\beta}}\;:\; u\in H^1_{loc}(\R,\R^+)\quad \text{and}\quad\int_\R u=1\right\}.$$
\end{rmk}
\begin{proof}
First notice that there exists a finite energy configuration, i.e. $w\in H^1(\R,\R)$ such that $F^\beta (w)<+\infty$. Indeed, every compactly supported and nonnegative $\mathcal{C}^1$ function has finite energy. Let take a minimizing sequence, i.e. $(w_n)_n\subset H^1(\R)\subset \mathcal{C}^0(\R)$ such that $F^\beta (w_n)\to c_\beta$. One can assume that $w_n$ is even and non-increasing on $\R^+$. Indeed if $w^\ast$ stands for the spherical rearrangement of some $w\in H^1(\R,\R^+)$, then one has $\int_\R |w^\ast|^\beta=\int_\R |w|^\beta$ and the classical Polya-Szego Theorem states that the spherical rearrangement reduces the Dirichlet energy of $w$:
\begin{equation*}
\label{PolyaSzego}
\int_{\R} |(w^*) '|^2 \leq \int_{\R} |w'|^2 \quad .
\end{equation*}
In particular $F^\beta (w^\ast)\leq F^\beta (w)$ as announced. Since $(w_n')_n$ is bounded in $L^2(\R)$, one can assume that it weakly converges in $L^2(\R)$. Moreover, as $w_n\geq 0$ a.e. and $\int_\R w_n=1$, $(w_n)_n$ is bounded in $L^1(\R)$. Thanks to the Poincar\'e-Wirtinger inequality, one deduces that $(w_n)_n$ is bounded in $H^1_{loc}(\R)$. Up to extraction, one can assume that $(w_n)_n$ weakly converges in $H^1_{loc}(\R)$. Let call $w\in H^1_{loc}(\R)$ the limit. In particular, $(w_n)_n$ strongly converges to $w$ in $L^1_{loc}(\R)$ and so $w$ is even, nonnegative and non-increasing on $\R^+$. Moreover, the Fatou lemma and the weak convergence of $w_n'$ yields
$$F^\beta (w)\leq\liminf\limits_{n\to\infty} F^\beta(w_n)=c_\beta\quad .$$
In order to prove that $w$ is a global minimizer it remains to prove that $w$ satisfies the constraint $\int w=1$. Indeed, from the Fatou Lemma we can only deduce that $\int w\leq 1$. One has to prove the strong convergence of $w_n$ in $L^1(\R)$. Since $w_n$ converges in $L^1_{loc}$, it is enough to prove that the sequence $(w_n)$ is tight. Let $R>0$. For all $n\geq 1$, since $w_n$ is non increasing on $[0,R]$, one has $\|w_n\|_{L^\infty (\{x\;:\; |x|>R)}=w_n(R)$ and Markov's inequality yields $w_n(R)\leq\frac{1}{2R}\int_\R w_n=\frac{1}{2R}$. Hence
\begin{equation*}
\int_{|x|>R}w_n(x)\d x\leq w_n(R)^{1-\beta} \int_\R w_n^\beta(x)\d x \leq (2R)^{\beta-1} F^\beta(w_n)\leq \frac{C}{R^{1-\beta}}\quad
\end{equation*}
for some constant $C>0$ non depending on $n$ which implies that $(w_n)_n$ is tight since $1-\beta>0$. Now, let check the regularity of $w$: Lipschitz continuous and smooth inside its support. Note that we already know that $w\in\mathcal{C}^{0,1/2}(\R)$ thanks to the Sobolev embedding $H^1_{loc}(\R)\subset\mathcal{C}^{0,1/2}(\R)$. Let check the regularity of $w$ inside its support. Since $w$ is even and non-increasing on $\R^+$, the set $\{x\;:\; w(x)>0\}$ is an interval $(-R,R)$ for some $R>0$. Since $w$ is a minimizer of the minimizing problem \eqref{hilliard}, $w$ also satisfies the following Euler-Lagrange equation
\begin{equation}
\label{EL_w}
\forall x\in (-R,R),\; -w''(x)+\beta w(x)^{\beta-1}=\lambda\quad ,
\end{equation}
where $\lambda\in\R$ is the Lagrange multiplier associated to the volume constraint $\int w=1$. Note that $\lambda=-w''(0)+\beta w(0)^{\beta-1}>0$. Indeed, since $x=0$ is a global maximum of $w$, $w'(0)=0$ and $w''(0)\leq 0$. From \eqref{EL_w}, one deduces that $w$ is smooth on $(-R,R)$. Now, multiplying \eqref{EL_w} by $w'$ and integrating it on $[0,x]$ or $[x,0]$ yields
$$\forall x\in (-R,R),\; \frac{w'(x)^2}{2}=w(x)^\beta-\lambda w(x)+\lambda w(0)-w(0)^\beta\quad .$$
Since $w$ is bounded on $\R$, we deduce that $w$ is Lipschitz continuous on $\R$. Last of all, we prove that $w$ is compactly supported, i.e. $R<\infty$. Assume by contradiction that $R=\infty$. Then $f\in\mathcal{C}^\infty(\R)$, $f>0$ on $\R$ and integrating \eqref{EL_w} yields
$$\forall x\in\R,\; w'(x)=\int_0^x (\beta w(y)^{\beta-1}-\lambda)\d y\quad .$$
Since $w(x)\underset{x\to\infty}{\longrightarrow}0$ and $\beta-1<0$, $w(x)^{\beta-1}\underset{x\to\infty}{\longrightarrow}+\infty$ and the right hand side of the preceding equation goes to $+\infty$ as well. Thus $w'(x)\underset{x\to\infty}{\longrightarrow}+\infty$ which is a contradiction.
\end{proof}
Now let $w_{\theta,\ve}$ be defined by $w_{\theta,\ve}(x)=\theta R_{\theta,\ve}^{-d}w (R_{\theta,\ve}^{-1} x)$ where $w$ is the optimal profile of Lemma \ref{chan} (satisfying all claimed regularity and symmetry properties) and let us introduce the kernel $\rho_{\theta,\ve}$ associated to $w_{\theta,\ve}$, given by the following lemma
\begin{lem}
There exists a bounded and compactly supported radial kernel $\rho_{\theta,\ve}\in L^\infty_c(\R^2,\R^+)$ such that $ w_{\theta,\ve}$ is the projection of $\rho_{\theta,\ve}$ on the axis $(x_1=0)$: 
$$\Pi^2_\sharp \rho_{\theta,\ve}(x)\d x=w_{\theta,\ve}(x_2)\d x_2,$$
where $\Pi^2$ stands for the projection on the second variable, $\d x$ (resp. $\d x_2$) is the Lebesgue measure on $\R^2$ (resp. $\R$) and $\Pi_\sharp\mu$ stands for the pushforward of some measure $\mu$ by $\Pi:\R^2\to\R$. Moreover, one can choose $\rho_{\theta,\varepsilon}$ of the form $\rho_{\theta,\varepsilon}(x)=R_{\theta,\varepsilon}^{-2}\rho(R_{\theta,\varepsilon}^{-1}x)$ for some $\rho\in L^\infty_c(\R,\R^+)$.
\end{lem}
\begin{proof}
After the renormalization $\rho_{\theta,\ve}(x)=R_{\theta,\ve}^{-2}\rho(R_{\theta,\ve}^{-1}x)$, it remains to find $\rho$ satisfying $\Pi^1_\sharp \rho(x)dx=w(x_2)\d x_2$. It is not very difficult to see that a radial solution is given by the formula
\begin{equation}\label{invproj}
\rho(x)=\displaystyle\int_{|x|}^{\infty}\frac{-w'(s)}{\pi\sqrt{s^2-|x|^2}}\; \d s.
\end{equation}
Details are left to the reader. Let justify how \eqref{invproj} implies that $\rho$ is bounded. Since $w$ is compactly supported, there exists $R>0$ such that $w(x)=0$ for $|x|>R$. In particular, $\rho$ is compactly supported on $\overline{B}(0,R)$. Then, since $w'(0)=0$ and since $w'$ is bounded on $\R$ and smooth around $0$, one has $|w'(x)|\leq C |x|$ for all $x\in\R$ and some $C>0$. Hence, there exists a constant $C>0$ such that for all $x\in\R^2$ such that $|x|=:r\in [0,R)$,
$$\rho (x)=\int_{1}^R \frac{-w'(rs)\d s}{\pi\sqrt{s^2-1}}\leq C\int_{1}^{R/r} \frac{rs\d s}{\sqrt{s^2-1}}\leq C \left\{R\int_1^2\frac{s\d s}{\sqrt{s^2-1}}+r\int_2^{R/r}\frac{s\d s}{\sqrt{s^2-1}}\right\}$$
which is bounded since $s\to \frac{s}{\sqrt{s^2-1}}$ is integrable on $[1,2]$ and bounded on $[2,+\infty)$.
\end{proof}
As a consequence, in the case where $u=\theta\hone_{|S}$, a recovery sequence, i.e. a sequence $(u_\ve)$ such that $u_\ve\to u$ in $\mathcal{M}_{div}(\Omega)$ and $M^\alpha_\ve (u_\ve)\to M^\alpha (u)$ as $\ve\to 0$, is obtained as
$$u_\ve =\rho_{\theta,\ve}\ast u.$$
In the case of a finite energy configuration, i.e. $u\in \mathcal{M}_{div}(\Omega)$ such that $M^\alpha (u)<\infty$, thanks to classical properties in the theory of $\Gamma$-convergence, it is enough to find a recovery sequence for $u$ belonging to a class of measures which are dense in energy. Thanks to the work of Q. Xia in \cite{Xia:2003} (see also \cite{Oudet:2011}), we know that the class of vector measures concentrated on finite graphs is dense in energy so that one can restrict to this case. In \cite{Xia:2003}, the branched transportation energy was in fact defined by relaxation of its restriction to the set of vector measures concentrated on a graph. The rectifiability of finite energy configurations and the Eulerian representation \eqref{defmalpha} was discovered later in \cite{Xia:2004} (see also \cite{Bernot:2009}). This density property was used in \cite{Oudet:2011} to prove the $\Gamma$-convergence of $M^\alpha_\ve$ toward $M^\alpha$. In the setting of functionals with divergence constraint, we need the following lemma:
\begin{lem}\label{lembd}
Let $u\in \mathcal{M}_{div}(\Omega)$ be such that $M^\alpha (u)<\infty$. For all $\lambda>\gamma$, there exists a sequence $(u_\ve)\subset H^1_0(\Omega)$ converging to $u$ in $\mathcal{M}_{div}(\Omega)$ such that 
$$M^\alpha_\ve (u_\ve)\underset{\ve\to 0}{\longrightarrow} c_0 M^\alpha (u)\quad\text{and}\quad\ve^{\lambda} \|\nabla\cdot u_\ve\|_{L^2}\quad\text{is bounded.}$$
\end{lem}
Before proving this statement, we are going to investigate the case where $u$ is concentrated on a finite graph. First of all, let us give some details on what ``a vector measure concentrated on a finite graph $G$'' is. Let $G=(V(G),E(G),\theta)$ be a weighted directed graph: $V(G)\subset \Omega$ is a finite set of vertices, $E(G)$ is the finite set of oriented edges $\mathbf{e}=(e,\tau_\mathbf{e})$, where $e=[a_e,b_e]\subset \Omega$ and $\tau_\mathbf{e}$ is a unit vector representing the direction of $\mathbf{e}$, and $\theta: E(G)\to (0,+\infty)$ is the weight function. Then the ``vector measure associated to $G$'' is given by
$$u_G=\sum_{\mathbf{e}=(e,\tau_{\mathbf{e}})\in E(g)} \theta(\mathbf{e})\tau_{\mathbf{e}} \d \hone_{|e}\quad .$$
These measures $u_G$ belong to $\mathcal{M}_{div}(\Omega)$, i.e. $\nabla\cdot u_G$ is a measure, and they are called ``transport paths'' (see Definition 2.1 in \cite{Xia:2003}). When $u$ is a transport path, we have the following lemma:
\begin{lem}\label{limsup}
Let $u=u_G\in\mathcal{M}_{div}(\Omega)$ for some weighted directed graph $G$. Then, there exists a sequence $(u_\ve)_{\ve>0}$ converging to $u$ in $\mathcal{M}_{div}(\Omega)$ and a constant $C$ depending on $u$ such that, for $\ve$ small enough, $u_\ve\in H^1_0(\Omega)$ and
\begin{enumerate}
\item $\int_\Omega |u_\ve|\leq |u|(\Omega)+ C\, \ve^\gamma$,
\item $\int_\Omega|\nabla\cdot  u_\ve |\leq |\nabla\cdot u|(\Omega)$,
\item $\ve^\gamma\, \|\nabla\cdot  u_\ve \|_{L^2} \leq C$,
\item $|M^\alpha_\ve(u_\ve) -c_0 M^\alpha (u)|\leq C\ve^\gamma$.
\end{enumerate}
\end{lem}

\begin{proof}
By definition, such a vector measure $u$ can be written as a finite sum of measures $u_i=\theta_i \,\tau_i\, \hone_{|S_i}$ concentrated on a segment $S_i\subset\Omega$ directed by $\tau_i$ with multiplicity $\theta_i$ for $i=1,\dots, I$. 
We first define a regularized vector fied $v_\ve$ by $v_\ve := \sum_i v_i$, where $v_i=\rho_{\theta_i,\ve}\ast u_i$. Then, for $\ve$ small enough, $v_\ve$ is compactly supported on $\Omega$ and satisfies
$$\begin{cases}
 \int_\Omega |v_\ve|\leq |u|(\Omega),&\\
 |M_\ve^\alpha(v_\ve) -c_0M^\alpha (u)|\leq C\ve^\gamma .&
\end{cases}$$
The first statement is a consequence of the fact that $\int \rho_{\theta_i,\ve}=1$ and the inequality $\|f\ast \mu\|
_{L^1}\leq \|f\|_{L^1}\,|\mu|(\Omega)$ for $f\in \mathcal{C}_c (\Omega)$ and for a finite measure $\mu$ on $\Omega$. For the second statement, by definition of the kernel $\rho_{\theta,\ve}$ we know that, out of the nodes
set $N=\bigcup_i\op{supp}(\nabla\cdot  v_i)$, 
$$M^\alpha_\ve (v_\ve ,N^c)=c_0 M^\alpha (v,N^c).$$ 
As a result, we just have to estimate these energies on $N$ which is a finite union of balls: the supports of $\rho_{\theta_i,\ve}$ recentered at each end-point of the segment $S_i$. Since the radius of these balls is of the order of $\ve^\gamma$, this immediately gives the fact that $M^\alpha (u,N)\leq C \ve^\gamma$ for some constant $C>0$ depending on $u$. For the sake of simplicity, in the rest of this proof, $C>0$ will denote some constant depending on $u$ which is large enough so that all the inequalities below are true. We are going to prove that
$$M^\alpha (u,N)+M^\alpha_\ve (v_{\ve},N)\leq C \ve^\gamma.$$
It remains to estimate $M^\alpha_\ve (v_{\ve},N)$. Since $M^\alpha_\ve (v_{\ve},N)\leq I \sum_i M^\alpha_\ve (v_i,N)$, it is enough to estimate $M^\alpha_\ve (v_i,N)$. But $\|v_i\|_{L^\infty (N)}=\|\rho_{\theta_i,\ve}\ast u_i\|_{L^\infty (N)}\leq C\ve^{-2\gamma} \|\rho\|_{L^\infty} |u_i|(N_i)$, where $N_i:=N+\op{supp}\rho_{\theta_i,\ve}:= \{x+y\;:\; x\in N,\, y\in \op{supp}\rho_{\theta_i,\ve}\}$. Note that $\op{supp}\rho_{\theta_i,\ve}$ is a ball centered at the origin with radius smaller than $C\ve^\gamma$ so that $N_i$ is a finite union of balls with radii smaller than $C\ve^\gamma$ as well. In particular, using the fact that $u_i=\theta_i \,\tau_i\, \hone_{|S_i}$, we get $|u_i|(N_i)\leq C \ve^\gamma$ and so $\|v_i\|_{L^\infty (N)}\leq C\ve^{-\gamma}$. Similarly, one has $\|\nabla v_i\|_{L^\infty (N)}=\|\nabla\rho_{\theta_i,\ve}\ast u_i\|_{L^\infty (N)}\leq C\ve^{-3\gamma} \|\nabla\rho\|_{L^\infty} |u_i|(N_i)\leq C\ve^{-2\gamma}$. Now, the definition \eqref{malphaepsilon2d} gives
$$M^\alpha_\ve (v_\ve, N)=\ve^{\alpha+1}\int_N |\nabla v_\ve|^2+\ve^{\alpha-1}\int_N |v_\ve|^\beta\leq |N|\{\ve^{\alpha+1} \|\nabla v_\ve\|_{L^\infty}^2+\ve^{\alpha-1} \|v_\ve\|_{L^\infty}^\beta\}.$$
From the inequality $|N|\leq C \ve^{2\gamma}$ and the equalities $\alpha=3\gamma -1$, $\beta\gamma=\frac{(4\alpha-2)\gamma}{\alpha+1} = 4\gamma-2$, we deduce
$$M^\alpha_\ve (v_\ve, N)\leq C\ve^{2\gamma}\{\ve^{3\gamma-4\gamma}+\ve^{3\gamma-2-\beta\gamma}\}\leq 2C\ve^\gamma$$
as required. In order to construct an approximating vector field with controlled divergence, we need to consider $u_\ve:=v_\ve-w_\ve$ where $w_\ve\in H^1_0(N)$ is constructed as follows:

The node set $N$, defined above, is a finite union $N=\bigcup_{j=1}^n B_j$, where each node $B_j$ is a ball centered at the end-point $a_i$ of some segment $S_i=[a_i,b_i]$. Let assume that $\ve$ is small enough so that these balls are non-overlapping. Then, on each node $B_j$, $g_j:=\nabla \cdot v_\ve$ is a finite superposition of kernels like $\rho_{\theta,\ve}$ recentered at $c_j$, the center of $B_j$. In particular $\|g_j\|_{L^2(B_j)}\leq C\ve^{-\gamma}$ and $\int_{B_j} g_j=\int_{B_j} \nabla\cdot v_\ve =(\nabla\cdot u)(B_j)=:\theta_j$. If $\theta_j=0$, then Theorem \ref{dirichlet} allows to find $w_j\in H^1_0(B_j)$ satisfying $\nabla \cdot w_{j} =g_j$ and $\| w_{j}\|_{H^1(B_j)}\leq C\, \ve^{-\gamma}$. If $\theta_j\neq 0$, let say $\theta_j>0$, we rewrite $g_j$ as $g_j=g^+-g^-=\lambda g^+ +(1-\lambda) g^+ -g^-$ where $g^+$ (resp. $g^-$) stands for the positive part (resp. negative part) of $g$ and $\lambda \in (0,1]$ is chosen such that $(1-\lambda)\int_B g^+ =\int_B g^-$, i.e. $\theta_j=\lambda \int_{B_j}g_+$. Applying Theorem \ref{dirichlet}, we get $w_j\in H^1_0(B_j)$ satisfying $\nabla \cdot w_{j} =(1-\lambda) g^+ -g^-$ and $\| w_{j}\|_{H^1(N)}\leq C\, \ve^{-\gamma}$. Let us define $w_\ve =\sum_j w_j$ and $u_\ve:= v_\ve - w_\ve$. Since $\int_{B_j} |\nabla \cdot u_\ve |=\int_{B_j} |g_j -\nabla \cdot w_j |= \lambda\int_{B_j}  g^+ =\theta_j$ for all $j$, we have 
$$\int_\Omega |\nabla\cdot u_\ve |=\int_N |\nabla\cdot u_\ve|\leq \sum_j \theta_j = |\nabla\cdot u |(\Omega).$$ 
Moreover, to estimate $\|\nabla\cdot u_\ve\|_{L^2}$, note that $\|\nabla\cdot  w_\ve\|_{L^2}\leq \| w_\ve\|_{H^1}\leq C\ve^{-\gamma}$ and, because $\nabla\cdot  v_\ve$ is only composed of a finite sum of translated kernels of the form $\rho_{\theta_i,\ve}$, $\|\nabla\cdot v_\ve\|_{L^2}\leq C\ve^{-\gamma}$ as well. In particular $\ve^{\gamma}\|\nabla\cdot u_\ve\|_{L^2}$ is bounded. Then, the Poincar\'e inequality yields
$$\|w_\ve\|_{L^2}= \sum_j \|w_j\|_{L^2(B_j)} \leq C\sum_j \ve^{\gamma}\|\nabla w_j \|_{L^2(B_j)}\leq C'$$
since $B_j$ is a ball of radius $C\varepsilon^\gamma$. Consequently, by the Cauchy-Schwarz inequality, we get
$$\int_\Omega |u_\ve|\leq \int_\Omega |v_\ve|+\int_N |w_\ve|\leq |u|(\Omega)+|N|^{1/2}\|w_\ve\|_{L^2}\leq |u|(\Omega)+C\ve^\gamma .$$
Similarly, by a H\"older inequality, we have
$$\int_N |w_\ve|^\beta \leq |N|^{\frac{2-\beta}{2}}\|w_\varepsilon\|_{L^2}^\beta\leq (n\varepsilon^{2\gamma})|^{\frac{2-\beta}{2}}\|w_\varepsilon\|_{L^2}^\beta\leq C\ve^{\gamma (2-\beta)}.$$
Once again, since $\alpha=3\gamma -1$ and $\beta\gamma= 4\gamma-2$, we deduce
$$M^\alpha_\ve(w_\ve)=\ve^{\alpha+1}\int_N |\nabla w_\ve|^2+\ve^{\alpha-1}\int_N |w_\ve|^\beta\leq C\{\ve^{\alpha+1-2\gamma}+\ve^{\alpha-1+\gamma (2-\beta)}\}=2C \ve^\gamma .$$ 
Since $M^\alpha (u,N)\leq C\ve^\gamma$, we get $M^\alpha_\ve(u_\ve,N)\leq 2[ M^\alpha_\ve(v_\ve,N)+M^\alpha_\ve(w_\ve,N)]\leq C\ve^\gamma$ which finally gives
$$|M^\alpha_\ve(u_\ve) -c_0 M^\alpha (u)|=|M^\alpha_\ve(u_\ve,N) -c_0 M^\alpha (u,N)|\leq C\ve^\gamma .\qedhere$$
\end{proof}

\begin{proof}[Proof of Lemma \ref{lembd}]
First fix $u\in \mathcal{M}_{div}(\Omega)$ and construct a sequence $(u_n)_{n\geq 1}$ converging to $u$ such that $u_n=u_{G_n}$ is a vector measure associated to some weighted directed graph $G_n\subset\Omega$ and $M^\alpha (u_n)$ converges to $M^\alpha (u)$. Since $(u_n)$ weakly converges in $\mathcal{M}_{div}(\Omega)$, the total variations of both measures $u_n$ and $\nabla\cdot  u_n$ are bounded by some constant $M>0$. In the following, we use a metric $d$ on the space $\mathcal{M}_{M+1}(\Omega)$. Extracting a subsequence if necessary, one can suppose that the two following estimates hold
$$d(u_n,u)\leq 2^{-n-1}\quad\text{and}\quad |M^\alpha (u_n)-M^\alpha (u)|\leq 2^{-n-1}.$$
For each $n\geq 1$, let $u_{\ve,n}$ be a sequence converging to $u_n$ as $\ve\to 0$ and satisfying all properties in Lemma \ref{limsup} for some constant $C=C_n$. Then, one can construct by induction a decreasing sequence $(\ve_n)_{n\geq 1}\to 0$ such that for all $n\geq 1$ and $\ve\leq\ve_n$, $u_{\ve,n}\in H^1_0(\Omega)$ and
\begin{enumerate}
\item $u_{\ve,n}\in \mathcal{M}_{M+1}(\Omega),$
\item $d(u_{\ve,n},u_n)\leq 2^{-n-1},$
\item$ |M^\alpha_\ve (u_{\ve,n})-c_0 M^\alpha (u_n)|\leq 2^{-n-1} ,$
\item$ \ve^{\lambda-\gamma}C_n\leq 1\text{ so that }\ve^{\lambda}\,\|\nabla\cdot  u_{\ve,n}\|_{L^2}\leq 1.$
\end{enumerate}
Indeed, assume that $\varepsilon_n>0$ satisfies all the asked properties. Then, one can find $\varepsilon_{n+1}\in (0,\varepsilon_n)$ small enough so that 

\begin{itemize}
\item[$\ast$] $C_{n+1}\,\varepsilon_{n+1}^\gamma<2^{-n-2}$ thus implying the first and third properties (see properties 1., 2. and 4. in Lemma \ref{limsup}),
\item[$\ast$] $C_{n+1}\,\ve_{n+1}^{\lambda-\gamma}<1$ which is possible since $\lambda>\gamma$
\item[$\ast$] and $d(u_{\ve,n+1},u_{n+1})\leq 2^{-n-2}$ for all $\varepsilon\in (0,\varepsilon_{n+1})$ which is possible since $u_{\varepsilon,n+1}$ converges to $u_{n+1}$ in $(\mathcal{M}_{M+1}(\Omega),d)$ as $\varepsilon\to 0$.
\end{itemize}
Now it is quite straightforward that the sequence $(u_\ve)_{\ve>0}$ defined by
\begin{equation*}
u_\ve =
\begin{cases}
u_{\ve,1} & if $\ve>\ve_1$,\\
u_{\ve,n} & if $\ve_{n+1}<\ve\leq\ve_{n}$ for some $n\geq 1$,
\end{cases}
\end{equation*}
satisfies all the properties of Lemma \ref{lembd}.
\end{proof}

\begin{proof}[Proof of Theorem \ref{gammacv}]
It is already shown in \cite{Oudet:2011} that $M^\alpha_\ve\overset{\Gamma}{\longrightarrow} c_0 M^\alpha$. We just have to prove that the $\Gamma -\limsup$ property still holds when we add the divergence constraint. In other words, it remains to prove that for all $u\in \mathcal{M}_{div}(\Omega)$ such that $\nabla\cdot u=\mu$, there exists a sequence $(v_\ve)_{\ve>0}\subset \mathcal{M}_{div} (\Omega)$ weakly converging to $u$ as measures, satisfying $\nabla\cdot v_\ve =f_\ve$ (so that $(v_\ve)$ also converges in $\mathcal{M}_{div}(\Omega)$) and $M^\alpha_\ve (v_\ve)\underset{\ve\to 0}{\longrightarrow} c_0 M^\alpha (u)$.\\

First of all, take a sequence $(u_\ve)_{\ve>0}\subset H^1_0(\Omega)$ converging to $u$ given by Lemma $\ref{lembd}$ for $\lambda=\frac{5\gamma}{4}$: $M^\alpha_\ve (u_\ve)\to M^\alpha (u)$ with $\ve^{\lambda} \|\nabla\cdot u_\ve\|_{L^2}$ bounded. Then define $g_\ve:=f_\ve-\nabla\cdot  u_\ve$ the residual divergence. Note that $\int_\Omega g_\varepsilon=0$. Indeed $\int_\Omega f_\varepsilon=0$ by assumption (see \eqref{hypponfeps}) and $\int_\Omega\nabla\cdot u_\varepsilon=0$ since $u_\varepsilon\in H^1_0(\Omega)$. Moreover, our assumptions on the sequences $(u_\varepsilon)_{\varepsilon>0}$ and $(f_\varepsilon)_{\varepsilon>0}$ yield
$$\ve^{\gamma_2}\|g_\ve\|_{L^2}^2\underset{\ve\to 0}{\longrightarrow}0\quad .$$ 
Indeed, we know that the same estimate is satisfied by $(f_\varepsilon)_{\varepsilon>0}$ thanks to \eqref{hypponfeps}. Moreover, since $3\gamma=\gamma_2$ and $\lambda=\frac{5\gamma}{2}$, one has $\varepsilon^{\gamma_2}\|\nabla\cdot u_\varepsilon\|_{L^2}=\varepsilon^{\frac{\gamma}{2}}\varepsilon^{\frac{5\gamma}{2}}\|\nabla\cdot u_\varepsilon\|_{L^2}\leq C\varepsilon^{\frac{\gamma}{2}}\to 0$. Moreover, since $f_\ve$ and $\nabla\cdot u_\ve$ weakly converge to $\mu$ as $\ve$ goes to $0$, we know that $g_\ve$ weakly converges to $0$. Let $g^+_\varepsilon$ (resp. $g^-_\varepsilon$) denote the positive (resp. negative) part of $g_\varepsilon$. In order to satisfy the divergence constraint, we may correct $u_\ve$ with a vector field $w_\ve$, given by Theorem \ref{comparisonwithwasserstein} (together with Remark \ref{dependancetheta}), such that $\nabla\cdot  w_\ve =g_\ve$,
\begin{equation}\label{fin}
M^\alpha_\ve (w_\ve)\leq H\left(W_{1}(g^+_\ve,g^-_\ve)^{1-d(1-\alpha)} + \ve^{\gamma_2}\|g_\ve\|_{L^2}^2\right)\underset{\ve\to 0}{\longrightarrow} 0
\end{equation}
and $\|w_\ve\|_{\mathcal{M}_{div}(\Omega)}$ is bounded, where $H(x)=C (x+x^\delta)$ for some $C>0$ and $\delta\in (0,1)$. We deduce that $(w_\ve)$ is relatively compact in $\mathcal{M}_{div}(\Omega)$. From \eqref{fin} and the $\Gamma -\liminf$ property, this implies that $w_\ve$ converges to $0$ in $\mathcal{M}_{div}(\Omega)$. Now, by construction, $v_\ve=u_\ve+w_\ve$ satisfies $\nabla\cdot  v_\ve = f_\ve$, $v_\ve \to u$ in $\mathcal{M}_{div}(\Omega)$ and $M^\alpha_\ve (v_\ve)\underset{\ve\to 0}{\longrightarrow} c_0 M^\alpha (u)$. Indeed, this last limit is a consequence of
\end{proof}

\begin{lem}Let $\Omega$ be some bounded open set in $\R^d$, $d\geq 1$. Let $(u_\ve)$, $(v_\ve)\subset H^1 (\Omega)$ be two sequences such that $M^\alpha_\ve (u_\ve-v_\ve)\underset{\ve\to 0}{\longrightarrow} 0$ and assume that $M^\alpha_\ve (v_\ve)$ is bounded. Then, 
$$|M^\alpha_\ve (u_\ve ) -M^\alpha_\ve (v_\ve)| \underset{\ve\to 0}{\longrightarrow} 0.$$
\end{lem}

\begin{proof}
Let $\nu>0$ be some constant. For all real matrices $A$ and $B$ of size $d\times d$, by the Young inequality, we have
$$|A+B|^2=|A|^2+|B|^2+2 A:B \leq (1+\nu)|A|^2+(1+1/\nu)|B|^2.$$
Writing $u_\ve=v_\ve+u_\ve-v_\ve$, we use the preceding inequality for $A=\nabla v_\ve$, $B=\nabla (u_\ve-v_\ve)$ and the subadditivity of $x\to|x|^\beta$ to get
$$M^\alpha_\ve (u_\ve)=\ve^{-\gamma_1}\int_\Omega |u_\ve |^\beta +\ve^{\gamma_2}\int_\Omega |\nabla u_\ve|^2
\leq (1+\nu)M^\alpha_\ve (v_\ve)+(1+1/\nu)M^\alpha_\ve (u_\ve-v_\ve).$$
Since $M^\alpha_\ve (v_\ve)<C$ for some constant $C<+\infty$, we deduce that
$$M^\alpha_\ve (u_\ve)-M^\alpha_\ve (v_\ve)\leq C\nu+(1+1/\nu) M^\alpha_\ve (u_\ve -v_\ve).$$
For any value of $\varepsilon$ such that $u_\varepsilon\neq v_\varepsilon$, let take $\nu = \sqrt{M^\alpha_\ve (u_\ve - v_\ve)}>0$. Hence, taking the $\limsup$ when $\ve\to 0$, one gets 
$$\limsup\limits_{\ve\to 0} \{M^\alpha_\ve (u_\ve)-M^\alpha_\ve (v_\ve)\}\leq C'\limsup\limits_{\ve\to 0} \sqrt{M^\alpha_\ve (u_\ve - v_\ve)}= 0.$$
Since $M^\alpha_\ve (v_\ve)\leq 2 [M^\alpha_\ve (u_\ve) +M^\alpha_\ve (v_\ve-u_\ve)]$ and $M^\alpha_\ve (v_\ve)$ is bounded, we deduce that $M^\alpha_\ve (u_\ve)$ is bounded as well. Then we can apply all the preceding computations exchanging $u_\ve$ and $v_\ve$ to get $\limsup\limits_{\ve\to 0} \{M^\alpha_\ve (v_\ve)-M^\alpha_\ve (u_\ve)\}\leq 0$ which concludes the proof.
\end{proof}

\paragraph{Aknowledgment.}
I would like to thank my advisor, Filippo Santambrogio, for discussions and advices which have proved to be very useful. This work was partially supported by the PGMO research project MACRO ``Mod\`eles d'Approximation Continue de R\'eseaux Optimaux''.
\bibliographystyle{plain}

\end{document}